\documentclass[11pt]{book}
\usepackage{geometry}                
\usepackage{graphicx}
\usepackage{amssymb}
\usepackage{amsfonts}
\usepackage{amsmath}
\usepackage{amsthm}

\usepackage{cancel}

\newcommand{\plim}{\varprojlim}
\newcommand{\leg}[2]{\left(\frac{#1}{#2}\right)}      

\newcommand{\Spin}{\mathrm{Spin}}   

\usepackage{xy}
\xyoption{all} 

\newcommand{\xysurj}{ \ar@{->>}[r]{&} }
\newcommand{\xyinj}{ \ar@{^{(}->}[r]{&} }

\usepackage[OT2,T1]{fontenc}
\DeclareSymbolFont{cyrletters}{OT2}{wncyr}{m}{n}
\DeclareMathSymbol{\Sha}{\mathalpha}{cyrletters}{"58}

\usepackage{enumerate}

\usepackage{framed}

\usepackage[refpage]{nomencl}

\makenomenclature

\def\@@@nomenclature[#1]#2#3{%
\def\@tempa{#2}\def\@tempb{#3}%
  \protected@write\@nomenclaturefile{}%
    {\string\nomenclatureentry{#1\nom@verb\@tempa @[{\nom@verb\@tempa}]%
      |nompageref{\begingroup\nom@verb\@tempb\protect\nomeqref{\theequation}}}%
      {\thepage}}%
  \endgroup 
  \@esphack}
\usepackage{makeidx}
\makeindex

\usepackage[pdfauthor={Jonathan Hanke},%
pdftitle={Quadratic Forms and Automorphic Forms},%
pagebackref=true,%
pdftex]{hyperref}
\hypersetup{
    bookmarks=true,         
    pdftoolbar=true,        
    pdfmenubar=true,        
    pdffitwindow=false,     
    pdfstartview={FitH},    
    pdftitle={Quadratic Forms and Automorphic Forms -- Expanded notes from the 2009 Arizona Winter School},    
    pdfauthor={Jonathan Hanke},     
    pdfcreator={Jonathan Hanke},   
    pdfproducer={Jonathan Hanke}, 
    pdfkeywords={automorphic form} {quadratic form} {theta}, 
    pdfnewwindow=true,      
    colorlinks=false,       
    linkcolor=red,          
    citecolor=green,        
    filecolor=magenta,      
    urlcolor=cyan           
}

\newcommand{\p}{\mathfrak{p}}        
\newcommand{\Q}{\mathbb{Q}}                 
\newcommand{\N}{\Bbb N}                 
\newcommand{\Z}{\Bbb Z}                 
\newcommand{\R}{\Bbb R}                 
\newcommand{\C}{\Bbb C}                 
\newcommand{\A}{\Bbb A}                  
\newcommand{\F}{\Bbb F}                  
\renewcommand{\O}{\mathcal{O}}     

\newcommand{\x}{\vec{x}}
\newcommand{\y}{\vec{y}}
\renewcommand{\v}{\vec{v}}
\newcommand{\w}{\vec{w}}

\newcommand{\bs}{\backslash}

\renewcommand{\(}{\left(}
\renewcommand{\)}{\right)}
\renewcommand{\[}{\left[}
\renewcommand{\]}{\right]}

\newcommand{\Char}{\mathrm{char}}
\renewcommand{\Im}{\mathrm{Im}}
\renewcommand{\Re}{\mathrm{Re}}
\renewcommand{\H}{\mathcal{H}}
\newcommand{\SL}{\mathrm{SL}}

\newcommand{\Cls}{\mathrm{Cls}}
\newcommand{\Gen}{\mathrm{Gen}}
\newcommand{\Spn}{\mathrm{Spn}}
\newcommand{\Aut}{\mathrm{Aut}}
\newcommand{\Stab}{\mathrm{Stab}}
\newcommand{\Mass}{\mathrm{Mass}}

\newcommand{\vol}{\mathrm{Vol}}

\newcommand{\scale}{\mathrm{Scale}}                       
\newcommand{\norm}{\mathrm{Norm}}                        

\newcommand{\B}{\mathcal{B}}                       

\newcommand{\pderiv}[1]{\frac{\partial}{\partial #1}}

\newcommand{\GL}{\mathrm{GL}}
\newcommand{\Sp}{\mathrm{Sp}}

\newcommand{\al}{\alpha}
\newcommand{\ra}{\rightarrow}

\newcommand{\ve}{\varepsilon}

\newcommand{\Proj}{\Bbb P}

\renewcommand{\a}{\mathfrak{a}}        
\newcommand{\n}{\mathfrak{n}}        

\newcommand{\id}{\mathrm{id}}

\newcommand{\Fcal}{\mathcal{F}}                  

\newcommand{\Ker}{\mathrm{Ker}}

\newtheorem{thm}{Theorem}[section]
\newtheorem{cor}[thm]{Corollary}
\newtheorem{lem}[thm]{Lemma}

\newtheorem{defn}[thm]{Definiton}

\theoremstyle{remark}
\newtheorem{rem}[thm]{Remark}

\newcommand{\K}{F}

\renewcommand{\a}{{\infty}}

\newcommand{\f}{{\bf f}}

\newcommand{\ord}{\text{ord}}

\newcommand{\levi}[1]{\begin{bmatrix} {#1} & 0 \\ 0 & {#1}^{-1}\end{bmatrix}}
\newcommand{\uni}[1]{\begin{bmatrix} 1 & {#1} \\ 0 & 1\end{bmatrix}}
\newcommand{\weyl}{\begin{bmatrix} 0 & 1 \\ -1 & 0\end{bmatrix}}
\newcommand{\W}{\mathcal{W}}

\newcommand{\quatalg}[3]{\(\frac{#1, #2}{#3}\)}

\renewcommand{\Char}{\mathrm{Char}}
\newcommand{\Sym}{\mathrm{Sym}}

\title{Quadratic Forms and Automorphic Forms}
\author{Jonathan Hanke}

\begin{document}
\maketitle

\tableofcontents

\part*{Revised Notes from the 2009 Arizona Winter School -- ``Quadratic Forms and Automorphic Forms''}


\printnomenclature[0.9in]  

{\bf Notation and Conventions}

\medskip
We let $\Z$, $\Q$, $\R$, $\C$ denote the usual integers, rational numbers, real numbers, and complex numbers, and also denote the natural numbers as $\N := \Z_{> 0} := \{1, 2, \cdots \}$.
We say that an $n \times n$ matrix $A = (a_{ij}) \in M_n(R)$ over a ring $R$ is {\bf even} if $a_{ii} \in 2R$, and {\bf symmetric} if $a_{ij} = a_{ji}$ for all $1\leq i,j \leq n$.  The symmetric matrices in $M_n(R)$ are denoted by $\Sym_n(R)$.
We denote the trivial (mod 1) Dirichlet character sending all integers to $1$ by $\mathbf{1}$. In analytic estimates, it is common to use the notation $X >> Y$ to mean that $X > C\cdot Y$ for some (implied) constant $C\in \R>0$.

Suppose $R$ is an integral domain and $V$ is a (finite dmensional) vectorspace over its field of fractions $F$.  
By a {\bf lattice} or {\bf $R$-lattice} in $V$ we will mean a finitely generated $R$-module over $R$ that spans $V$.  In particular, notice that we will always assume that our lattices have full rank in $V$.

If $F$ is a number field (i.e. a finite field extension of $\Q$), then we define a {\bf place} or {\bf normalized valuation} of $F$ to be an equivalence class of metrics $|\cdot |_v$ on $F$ that induce the same topology on $F$.  We implicitly identify $v$ with the distinguished metric in each class agreeing with the usual absolute value when $F_v = \R$ or $\C$, and giving $|\p|_v = |\O_v/\p\O_v|$ when $v$ is non-archimedean and $F_v$ has valuation ring $\O_v$ with maximal ideal $\p$.  If $p \in \N$ with $(p)=\p \cap \Z$ then we have an associated {\bf valuation} on $F_v^\times$ given by $\ord_v(\cdot) := -\log_p(\cdot) = \ord_\p(\cdot)$.
We define a {\bf (non-zero) squareclass} of a field $K$ to be an element of $K^\times/(K^\times)^2$, and when $K$ is a non-archimedean local field then the valuation $\ord_v$ descends to give a $(\Z/2\Z)$-valued valuation on squareclasses.

We say that an $R$-valued quadratic form over a (commutative) ring $R$ is {\bf primitive} if the ideal generated by its values $(Q(R))$ is $R$ and we say that {\bf $Q$ represents $m$} if $m\in Q(R)$.
Given a quadratic space $(V,Q)$ of dimension $n$ over a field $F$, we define the orthogonal group 
$O(V) := O_Q(V)$ to be the set of invertible linear transformations $L:V \ra V$ so that $Q(L(\v)) = Q(\v)$ for all $\v \in V$.  Given a basis of $V$, $O(V)$ can be realized as a subset of $\GL_n(F)$.  We also define the {\bf special orthogonal group} $SO(V) := O^+(V)$ as the (orientation-preserving) subgroup of $O(V)$ having determinant 1.


\bigskip
\bigskip


{\bf Dedication}  

\medskip
I would like to extend a special thanks to my advisor Goro Shimura, without whom I would not have become involved with this beautiful subject, and these notes would have not been possible.  His dedication to careful exposition and referencing have been a major influence on these notes, and hopefully this attention to detail will make it easier for the reader seeking to learn this material.  Having said this, despite my best efforts I am sure that these notes contain mistakes, and any corrections are very welcome.  Please feel free to send them to {\tt jonhanke@gmail.com}.

\medskip
I dedicate these notes to the many excellent expositors whose efforts have helped me to learn new areas of mathematics, and 
%
including (but not limited to) Tony Knapp, Steve Gelbart, Bill Duke, Henryk Iwaniec, and Serge Lang.

\chapter*{Preface}

These notes are an extension of the rough notes provided for my four lecture graduate level course on ``Quadratic Forms and Automorphic Forms'' at the March 2009 Arizona Winter School on Quadratic Forms.  They are meant to give a survey of some aspects of the classical theory of quadratic forms over number fields and their rings of integers (e.g. over $\Q$ and $\Z$), and their connection with modular and automorphic forms.  

Originally I had hoped to expand these notes to include many other interesting topics related to Clifford algebras and various automorphic ``liftings'' that are a natural outgrowth of this chain of ideas.  Due to practical time deadlines these notes are essentially a written version of the talks which have been filled out to include precise references for all theorems and details of less well-known arguments with the hope of enabling an eager graduate student to gain a working knowledge of the basic ideas and arguments for each of the topics covered.

I would like to thank the organizers of the 2009 Arizona Winter School (David Savitt, Fernando Rodriguez-Villegas, Matt Papanikolas, William Stein, and Dinesh Thakur) for the opportunity to give these lectures, as well as all of the students who worked very hard in the evenings to make progress on the research projects associated to these lectures.  Special thanks go to John Voight for his many hours helping these students, and also for all of the work he put in to helping to write up notes for Professor Conway's lectures.  Also several helpful comments made by Pete L. Clark, Danny Krashen, Rishikesh,  Robert Varley and the anonymous referee when proofreading of these notes.

I am grateful to MSRI for their hospitality and 24-hour library access during the Spring 2011 semester, during which the final version of these notes were written and many of the references were added.  I also acknowledge the support of NSF Grant DMS-0903401 during the Winter School and while these notes were being written.

For a more lively (but perhaps less precise) introduction to this material, the reader is encouraged to view videos of the lectures online at the Arizona Winter School webpage
$$
\text{\tt http://swc.math.arizona.edu/aws/09/}
$$
Any errata to the published version of these notes will be posted on my website
$$
\text{\tt http://jonhanke.com}
$$

%

%





\chapter{Background on Quadratic Forms}\label{Chapter:Background}

\section{Notation and Conventions}

In these notes we will study aspects of the theory of quadratic forms over rings $R$ 
\nomenclature{$R$}{a ring of characteristic $\neq 2$}
and fields $F$ 
\nomenclature{$F$}{a field of characteristic $\neq 2$}
of characteristic $\Char(\cdot)$
\nomenclature{$\Char(R)$}{the characteristic of the ring $R$}
$\neq 2$ (i.e. where $1 + 1 \neq 0$).  
%
%
%
While one can discuss quadratic forms in characteristic 2, we can no longer equate them with symmetric bilinear forms, so the theory there is more complicated.  For references valid in characteristic 2, we refer the reader to \cite{Kaplansky:2003kx, Knus:1991qa, Elman:2008fk, Bak:1981uq, Sah:1960kx}.  

Our main interest is in quadratic forms over the field $\Q$\nomenclature{$\Q$}{the rational numbers}, the ring $\Z$\nomenclature{$\Z$}{the rational integers}, and their completions, though we may consider a more general setting (e.g. a number field and its ring of integers) when there are no additional complications in doing so.

\section{Definitions of Quadratic Forms}

In this chapter we give some basic definitions and ideas used to understand quadratic forms and the numbers they represent.  We define a {\bf quadratic form $Q(\x)$ over a ring $R$} to be a degree 2 homogeneous polynomial 
$$
Q(\x) := Q(x_1, \cdots, x_n) := \sum_{1\leq i \leq j \leq n} c_{ij} x_i x_j
$$
\nomenclature{$\x, \y, \v$}{$n$-tuples of elements of a ring $R$}
in $n$ variables with coefficients $c_{ij}$ in $R$.

\medskip

When division by 2 is allowed (either in $R$ or in some ring containing $R$)
we can also
consider the quadratic form $Q(\x)$ as coming from the symmetric {\bf Gram bilinear form}
\begin{equation}
B(\x, \y) := \sum_{1 \leq i,j \leq n} b_{ij} x_i y_j = {}^t\x \, B \, \y 
\end{equation}
\nomenclature{${}^tA$}{the transpose of a matrix $A$}
via the formula $Q(\x) = B(\x, \x) = {}^t\x \, B \, \x$, where the matrix $B := (b_{ij})$ and $b_{ij} = \frac12(c_{ij} + c_{ji})$ (with the convention that $c_{ij} = 0$ if $i>j$).
We refer to the symmetric matrix $B = (b_{ij})$ as the {\bf Gram matrix} of $Q$.  It is common to relate a quadratic form $Q(\x)$ to its Gram bilinear form $B(\x,\y)$ by the {\bf polarization identity} 
\begin{align}
Q(\x + \y) 
&= B(\x + \y, \x + \y) \notag\\
&= B(\x, \x) + 2B(\x, \y) + B(\y, \y) \label{Eq:Polarization_id} \\
&= Q(\x) + 2B(\x, \y) + Q(\y).\notag
\end{align}
From either of these formulas for $B(\x,\y)$, we see that the matrix $B \in \frac{1}{2}\Sym_n(R)$.
\nomenclature{$\Sym_n(R)$}{the symmetric $n \times n$ matrices with coefficients in $R$}

Since often $\frac{1}{2} \notin R$, it is somewhat unnatural to consider the Gram bilinear form since it is not an object defined over $R$.  However the {\bf Hessian bilinear form} $H(\x, \y) := 2B(\x, \y)$ is defined over $R$ and can be seen to be very naturally associated to the quadratic form $Q(\x)$ by the polarization identity.   This definition is motivated by the fact that the matrix $H := 2B \in \Sym_n(R)$ of the Hessian bilinear form is the {\bf Hessian matrix} of second order partial derivatives of $Q(\x)$ (i.e. $H = (H_{ij})$ where $a_{ij} = \pderiv{x_i}\pderiv{x_j} Q(\x)$).  
For this reason, it is often preferable to use the Hessian formulation when working over rings, and in particular when $R  = \Z$.
Notice that the diagonal coefficients $h_{ii}$ of the Hessian matrix are even, so $H$ is actually an even symmetric matrix.  In the geometric theory of quadratic forms the Hessian bilinear form is often referred to as the {\bf polar form} of $Q$ because of its close connection with the polarization identity (see \cite[p39]{Elman:2008fk}).



\medskip
Another perspective on quadratic forms is to think of them as {\bf free quadratic $R$-modules}, by which we mean an $R$-module $M \cong R^n$ equipped with a ``quadratic function'' 
$$
Q: M \ra R
$$
which is a function satisfying 
\begin{enumerate}
\item $Q(a\x) = a^2 Q(\x)$ for all $a \in R$, and for which 
\item$H(\x, \y) := Q(\x + \y) - Q(\x) - Q(\y)$ is a bilinear form.  
\end{enumerate}
This perspective is equivalent to that of a quadratic form because one can recover the quadratic form coefficients from $c_{ii} = Q(\vec e_i)$, and when $i < j$ we have $c_{ij} = H(\vec e_i, \vec e_j)$.

In the case when $R = \Z$, this perspective allows us to think of $Q$ as a ``quadratic lattice'' which 
naturally sits isometrically in a vector space over $\Q$ that is also equipped with a quadratic form $Q$.
%
%
More precisely, we define a {\bf quadratic space} to be a pair $(V, Q)$ 
\nomenclature{$(V, Q)$}{a quadratic space, meaning a vector space $V$ over $F$ equipped with a quadratic form $Q$}
where $V$ is a finite-dimensional vector space over $F$ and $Q$ is a quadratic form on $V$.  We say that a module $M$ is a {\bf quadratic lattice} if $M$ is a (full rank) $R$-lattice in a quadratic space $(V,Q)$ over $F$, where $F$ is the field of fractions of $R$.  Notice that we can always realize a quadratic form $Q(\x)$ over an integral domain $R$ as being induced from a free quadratic lattice $R^n$ in some quadratic space $(V,Q)$, by thinking of $Q$ as a function on $V = F^n$ where $F$ is the fraction field of $R$.
%
%
%
%
{\it 
(Note: However if $R$ is not a principal ideal domain then there will exist non-free quadratic lattices, which cannot be described as quadratic forms!)
}

We say that a quadratic form $Q(\x)$ is {\bf $R$-valued} if $Q(R^n) \subseteq R$.  From the polarization identity we see that $Q(\x)$ is $R$-valued iff $Q(\x)$ is defined over $R$ (i.e. all coefficients $c_{ij} \in R$) because $c_{ij} = H(\vec e_i,\vec e_j)$ 
\nomenclature{$\vec e_i$}{the vector $(0, \cdots, 0, 1, 0, \cdots, 0)$ with 1 in the $i^\text{th}$-component}
when $i \neq j$ and $c_{ii} = Q(\vec e_i)$.

\section{Equivalence of Quadratic Forms}

Informally, we would like to consider two quadratic forms as being ``the same'' if we can rewrite one in terms of the other by a change of variables.  
%
More precisely, we say that two quadratic forms $Q_1$ and $Q_2$ are {\bf equivalent over $R$}, and write $Q_1 \sim_R Q_2$, if there is an invertible linear change of variables $\phi(\x)$ with coefficients in $R$ so that $Q_2(\x) = Q_1(\phi(\x))$.
We can represent $\phi(\x)$ (with respect to the generators $\vec e_i$ of the free module $R^n$) as left-multiplication by an invertible matrix $M$ over $R$ -- i.e. $M \in M_n(R)$ 
\nomenclature{$M_n(R)$}{the $n\times n$ matrices with coefficients in $R$}%
and
$$
\phi(\x) = M\x.
$$
Expressing this equivalence in terms of the associated Hessian and Gram matrices, we see that composition with $\phi$ gives the equivalence relations
\begin{equation} \label{eq:change_of_vars}
H_1 \sim_R
{}^tM H_1 M = H_2 
\qquad \text{ and } \qquad 
B_1 
\sim_R
{}^tM B_1 M  = B_2
\end{equation}
where
$$
Q_i(\x) = \tfrac{1}{2}{}^t\x H_i \x = {}^t\x B_i \x.
$$
For quadratic lattices, the corresponding notion is to say that  
two quadratic $R$-modules are equivalent if there is a $R$-module isomorphism between the modules commuting with the quadratic functions (i.e. preserving all values of the the quadratic function).  Because this $R$-module isomorphism preserves all values of the the quadratic function, it is often referred to as an {\bf isometry}.  We will see later that this idea of {\it thinking of equivalent quadratic lattices as isometric lattices in a quadratic space is very fruitful}, and can be used to give a very geometric flavor to questions about the arithmetic of quadratic forms.



\section{Direct Sums and Scaling}

Two important constructions for making new quadratic forms from known ones are the operations of scaling and direct sum.  Given $a\in R$ and a quadratic form $Q(\x)$ defined over $R$, we can define a new {\bf scaled quadratic form} $a\cdot Q(\x)$ which is also defined over $R$.  We can also try to detect if a quadratic form is a scaled version of some other quadratic form by looking at its values, generated either as a bilinear form or as a quadratic form.  We therefore define the {\bf (Hessian and Gram) scale} and {\bf norm} of $Q$ by 
\begin{align}
\scale_H(Q) &:= \{H(\x, \y) \mid \x, \y\in R^{n} \}\\
\scale_G(Q) &:= \{B(\x, \y) \mid \x, \y\in R^{n} \}\\
\norm(Q) &:= \{Q(\x) \mid \x\in R^{n} \}
\end{align}
and notice that $\scale_{H/G}(a\cdot Q) = a\cdot \scale_{H/G}(Q)$ and $\norm(a\cdot Q) = a^2\cdot \norm(Q)$.

Another useful construction for making new quadratic forms is to take the direct sum of two given quadratic forms.  Given $Q_{1}(\x_{1})$ and $Q_{2}(\x_{2})$ in $n_{1}$ and $n_{2}$ distinct variables over $R$ respectively, we define their {\bf (orthogonal) direct sum} $Q_{1} \oplus Q_{2}$ as the quadratic form 
$$
(Q_{1} \oplus Q_{2})(\x) := Q_{1}(\x_{1}) + Q_{2}(\x_{2}) 
$$
in $n_{1} + n_{2}$ variables, where $\x := (\x_{1}, \x_{2})$.



\section{The Geometry of Quadratic Spaces}

Quadratic forms become much simpler to study when the base ring $R$ is a field (which we denote by $F$) .   In that case we know that all finite-dimensional $R$-modules are free (i.e. every finite-dimensional vector space has a basis), and that there is exactly one of each dimension $n$.  This simplification allows us to replace commutative algebra with linear algebra when studying quadratic forms, and motivates our previous definition of a {\bf quadratic space} as pair $(V, Q)$ consisting of a quadratic form $Q$ on a finite-dimensional vector space $V$ over $F$.  We will often refer to any quadratic form (in $n$ variables) over a field as a quadratic space, with the implicit understanding that we are considering the vector space $V := F^n$.  We also often refer to an equivalence  of quadratic spaces (over $F$) as an {\bf isometry}.

We now present some very useful geometric classification theorems about quadratic spaces.  To do this, we define the {\bf (Gram) inner product} of $\v_1, \v_2 \in V$ as the value of the Gram symmetric bilinear form $B(\v_1, \v_2)$.
\footnote{We could also have defined an inner product using the Hessian bilinear form, but this choice is less standard in the literature and the difference will not matter for our purposes in this section.}
We say that two vectors $\v_1, \v_2 \in V$ are {\bf perpendicular} or {\bf orthogonal} and write $\v_1 \perp \v_2$ if their inner product $B(\v_1, \v_2) = 0$.  Similarly we say that a vector $\v$ is {\bf perpendicular to a subspace $W \subseteq V$} if $\v\perp \w$ 
for all $\w \in W$, 
and we say that two subspaces $W_1, W_2 \subseteq V$ are perpendicular if $\w_1 \perp \w_2$ for all $\w_i \in W_i$.

Our first theorem says we can always find an orthogonal basis for $V$, which we can see puts the (Gram/Hessian) matrices associated to $Q$ in diagonal form:

\begin{thm}\label{Thm:orthogonal_splitting}
[Orthogonal splitting/diagonalization]
Every quadratic space $V$ admits an orthogonal basis.
\end{thm}
\begin{proof}
This is proved in Cassels's book \cite[Lemma 1.4 on p13-14]{Cassels:1978aa}, where he refers to this as a ``normal basis''.  
\end{proof}

Given a quadratic space $(V,Q)$, to any a choice of basis $\B := \{\v_1, \dots, \v_n\}$ for $V$ we can associate a quadratic form $Q_\B(\x)$ by expressing elements of $V$ in the coordinates of $\B$:
$$
\B \quad\longmapsto\quad Q_\B(\x) := Q(x_1 \v_1 + \cdots x_n \v_n).
$$
We can use this association to define the {\bf determinant} $\det(Q)$ of $(V,Q)$ as the determinant $\det(B)$ of the Gram matrix $B$ of the associated quadratic form $Q_\B(\x)$ for some basis $\B$.  However since changing the basis $\B$ induces the equivalence relation (\ref{eq:change_of_vars}), we see that $\det(Q)$ is only well-defined up to multiplication by $\det(M)^2 \in K^\times$, and so $\det(Q)$ gives a well-defined square-class in $K/(K^\times)^2$.
%
%
We say that $Q$ is {\bf degenerate} if $\det(Q) = 0$, otherwise we say that $Q$ is {\bf non-degenerate} (or {\bf  regular}).  By convention, the zero-dimensional quadratic space has $\det(Q) = 1$ and is non-degenerate.

\begin{lem}
If a quadratic space $(V,Q)$ is degenerate, then there is some non-zero vector $\v \in V$ perpendicular to $V$ (i.e., $\v \perp V$).
\end{lem}

The next theorem states that we can always reduce a degenerate quadratic space to a non-degenerate space by (orthogonally) splitting off a {\bf zero space} $(V,Q)$, which we define to be a vector space $V$ equipped with the identically zero quadratic form $Q(\x)=0$.  Notice that a zero space has an inner product that is identically zero, so it is always perpendicular to itself.  (In the literature a zero space is often referred to as a {\bf totally isotropic space}.)  We define {\bf the radical of a quadratic space $(V,Q)$} as the maximal (quadratic) subspace of $V$ perpendicular to all $\v \in V$, which is just the set of vectors perpendicular to all of $V$.

\begin{lem}[Radical Splitting]
Every quadratic space can be written as a orthogonal direct sum of a zero space (the radical of the quadratic space) and a non-degenerate space.
\end{lem}
\begin{proof}
This is given in Cassels's book \cite[Lemma 6.1 on p28]{Cassels:1978aa}.
\end{proof}

We say that a non-zero vector $\v \in V$ is {\bf isotropic} if $Q(\v)=0$ and say that $\v$ is {\bf anisotropic} otherwise.  Extending this definition to subspaces, we say that a subspace $U \subseteq (V,Q)$ is isotropic if it contains an isotropic vector, and anisotropic otherwise.  Notice that $U$ is a totally isotropic subspace $\iff$ every non-zero vector in $U$ is isotropic.  

For non-degenerate quadratic spaces, isotropic vectors play a key role because of their close relation to the {\bf hyperbolic plane} $H_2$, which is defined as the two-dimensional quadratic space (say with coordinates $x$ and $y$) endowed with the quadratic form $Q(\x) = Q(x, y) = xy$.  We also refer to the orthogonal direct sum of $r$ hyperbolic planes as the {\bf hyperbolic space} $H_{2r}$, which has dimension $2r$.



\begin{thm}[Totally Isotropic Splitting]
Suppose $(V,Q)$ is a non-degenerate quadratic space.  Then for every $r$-dimensional zero subspace $U \subseteq V$ we can find a complementary $r$-dimensional subspace $U' \subseteq V$ so that $V = U \oplus U' \oplus W$ as vector spaces, and (as quadratic subspaces) $W$ is non-degenerate and $U \oplus U'$ is equivalent to the hyperbolic space $H_{2r}$. 
\end{thm}
\begin{proof}
This is shown in Lam's Book \cite[Thrm 3.4(1--2), p10]{Lam:2005kl}, or by repeated application of \cite[Cor 1, p15]{Cassels:1978aa}.
\end{proof}

In the case that $W$ is isotropic, we can repeatedly apply this to split off additional hyperbolic spaces until $W$ is anisotropic.  So we could have initially taken $U$ to be a totally isotropic subspace of maximal dimension in $(V,Q)$, called a {\bf maximal totally isotropic subspace}, and then concluded that $W$ was anisotropic.


\medskip


Another particularly useful result about quadratic spaces is that there is a large group of $F$-linear isometries of $(V,Q)$, called the  {\bf orthogonal group} and denoted as $O_Q(V)$ or $O(V)$, that acts on $(V,Q)$.  We will see throughout these lectures that the orthogonal group is very closely connected to the arithmetic of quadratic forms, partly because of the following important structural theorem of Witt which classifies isometric quadratic subspaces within a given quadratic space in terms of the orbits of $O(V)$.

\begin{thm}[Witt's Theorem]
Suppose that $U$ and $U'$ are non-degenerate isometric (quadratic) subspaces of a quadratic space $V$.  Then any isometry $\al:U \ra U'$ extends to an isometry $\al:V\ra V$. 
\end{thm}

\begin{proof}
This is proved in almost every quadratic forms book, e.g.
Cassels's book \cite[Thrm 4.1 on p21]{Cassels:1978aa}, Shimura's book \cite[Thrm 22.2 on p116-7]{Shimura:2010vn}, and Lam's book \cite[Thrm 4.2 and 4.7 on pp 12-15]{Lam:2005kl}.
\end{proof}

Notice that Witt's Theorem shows that the dimension of a maximal isotropic subspace of a quadratic space $(V,Q)$ is a well-defined number (independent of the particular maximal isotropic subspace of $V$ that we choose), since otherwise we could find an isometry of $V$ that puts the smaller subspace properly inside the larger one, violating the assumption of maximality.







\section{Quadratic Forms over Local Fields} \label{Sec:Local_quadratic_spaces}

We now suppose that $(V,Q)$ is a non-degenerate quadratic space over one of the local fields $F=\R, \, \C$ or $\Q_p$ where $p$ is a positive prime number.  In this setting we can successfully classify quadratic spaces in terms of certain invariants associated to them.  The major result along these lines is that in addition to the dimension and determinant, at most one additional invariant is needed to classify non-degenerate quadratic spaces up to equivalence.

\begin{thm} 
There is exactly one non-degenerate quadratic space over $\C$ of each dimension $n$.  
\end{thm}

\begin{proof}
From the orthogonal splitting theorem \ref{Thm:orthogonal_splitting} we see that any such $Q(\x) \sim_\C \sum_{i=1}^n c_i x_i^2$ with $c_i \in \C^\times$. However since $\C^\times = (\C^\times)^2$ every $c_i$ can be written as some $a_i^2$, so we see that 
$Q(\x) = \sum_{i=1}^n (a_i x_i)^2 \sim_\C \sum_{i=1}^n x_i^2$.
\end{proof}

\begin{thm}
The non-degenerate quadratic spaces over $\R$ are in 1-1 correspondence with the pairs $(n, p_1)$ where $0\leq p_1 \leq n$.
\end{thm}

\begin{proof}
Since $\R^{\times}$ has two squareclasses $\pm (\R^{\times})^{2}$, we see that the diagonal elements can be chosen to be either $1$ or $-1$.  Since the dimension of a maximal totally isotropic subspace is a well-defined isometry invariant, this characterizes the number of $(1,-1)$ pairs on the diagonal, and then the remaining diagonal entries all have the same sign.  Its orthogonal complement is anisotropic, and is either $1_{n-2r}$ or $-1_{n-2r}$ depending on the sign of the values it represents.
\end{proof}

In practice it 
is more standard to use the {\bf signature invariant} $p :=p_{1} - p_{2}$ 
instead of $p_1$, however they are equivalent for our purpose of giving a complete set of invariants for quadratic spaces over $\R$.


\begin{thm}
The non-degenerate quadratic spaces over  $\Q_p$ are in 1-1 correspondence with the triples $(n, d, c)$ where $n = \dim(Q) \in \Z\geq0$, $d = \det(Q) \in \Q_p^\times/(\Q_p^\times)^2$, and $c \in \{\pm 1\}$ is the Hasse invariant of $Q$, under the restrictions that 
\begin{enumerate}
\item$c = 1$ if either $n=1$ or $(n, d) =  (2, -1)$, and also 
\item $(n,d,c) = (0,1,1)$ if $n=0$.
\end{enumerate}
\end{thm}
\begin{proof}
This is proved in Cassels's book \cite[Thrm 1.1 on p55]{Cassels:1978aa}.
\end{proof}

Another area in which we can extract more information about quadratic forms over local fields is in terms of the {\bf maximal anisotropic dimension of $K$}, which is defined to be the largest dimension of an anisotropic subspace of any quadratic space over $K$.  This is sometimes called the {\bf $u$-invariant} of $K$.

\begin{thm}
The maximal anisotropic dimensions of $\C, \R$ and $\Q_p$ are $1, \infty$, and 4. 
\end{thm}

\begin{proof}
Over $\C$ any form in $n\geq 2$ variables is isotropic, and any non-zero form of dimension 1 is anisotropic.  Over $\R$ we see that $Q(\x) = \sum_{i=1}^n x_i^2$ is anisotropic for any $n \in \N$.  Over $\Q_p$ any form in $\n\geq 5$ variables isotropic, and there is always an anisotropic form in four variables arising as the norm form from the unique ramified quaternion algebra over $\Q_p$ (see \cite[Thrm 2.12 and Cor 2.11, p158]{Lam:2005kl}).
\end{proof}

%
In particular, over $\R$ says that it is possible for quadratic forms of any dimension to be anisotropic, which means that either $Q$ represents only positive or only negative values (using a non-zero vector), and in these cases we say $Q$ is {\bf positive definite} or {\bf negative definite} respectively.  When a non-degenerate $Q$ is isotropic over $\R$ then it must represent both positive and negative values, in which case we say that $Q$ is {\bf indefinite}. 
%
%
For $\Q_p$ there is no notion of positive and negative, but one can concretely understand the $u$-invariant $u(\Q_{p}) = 4$ from the existence of certain (non-split) quaternion algebras at every prime $p$, which will be discussed in more detail in Chapter \ref{Chapter:Clifford}.  The norm forms of these quaternion algebras assume all values of $F$ and do not represent zero non-trivially.





\section{The Geometry of Quadratic Lattices -- Dual Lattices}

Quadratic lattices also have a kind of geometry associated to them that is a little more subtle than the ``perpendicular'' geometry of subspaces of quadratic spaces.  Given a (full rank) quadratic lattice $L$ (over $R$) in a quadratic space $(V,Q)$ (over the fraction field $F$ of $R$), we can consider the elements of $V$ where the linear form $H(\cdot, L)$ is in any fixed ideal $I$ of $R$.  
\footnote{For any subset $S \subseteq V$ the set $H(S, L)$ is an $R$-module, and so it is natural to consider maximal subsets of $V$ where $H(S, L)$ is a fixed $R$-module.  From the bilinearity of $H$, we see that these maximal sets $S$ are also $R$-modules.}
When $R=F$ is a field,
the only ideals are $F$ and $(0)$; 
taking $I = F$ imposes no condition,
and taking $I =(0)$ recovers the notion of the orthogonal complement $L^\perp$ of $L$.  However when $R \neq F$ 
then taking $I=R$ gives 
an
interesting integral notion of ``orthogonality'' which is very useful for making other lattices that are closely related to $L$.


We define the 
{\bf (Hessian) dual lattice} $L^{\#}$ of $L$ to be the set of vectors in $V$ that have integral inner product with all $\w\in L$, i.e.
$$
L^{\#} := \{\v \in V \mid H(\v, \w) \in R \text{ for all } \w \in L\}.
$$
Notice that if $H(\x, \y)$ is $R$-valued for all $\x,\y\in L$ then we have $L \subseteq L^{\#}$.  When an $R$-valued lattice $L$ is free as an $R$-module, 
we also know that 
the matrix of $H$ in any basis of $L$ is symmetric with coefficients in $R$ and even diagonal (i.e. all $a_{ii}\in 2R$), 
so $H$ is an even symmetric matrix.  
We then define the {\bf level} of $L$ to be the smallest (non-zero) ideal $\n \subseteq R$
so that the matrices in $\n H^{-1}$ are also even.  The level is a very useful invariant of $L$ which appears when we take dual lattices (because $H^{-1}$ is the matrix of basis vectors for the dual basis of the given basis $\B$ of $L$ in the coordinates of $\B$), and in particular it plays an important role in the theory of theta functions (see Chapter \ref{Chapter:ThetaFunctions}).

In the special case where $R = \Z$ the level $\n$ can be written as $\n = (N)$ for some $N \in\N$, and this (minimal) $N$ is what is usually referred to as the level of the quadratic lattice (which is also a quadratic form since all $\Z$-lattices are free).

We say that a quadratic form over a ring $R$ is  {\bf (Hessian) unimodular} if its Hessian bilinear form has unit determinant (i.e. $\det(H) \in R^\times$).  In terms of quadratic lattices, this is equivalent to saying that the associated quadratic $R$-lattice $(L, Q)= (R^n, Q)$ is {\bf (Hessian) self-dual} (i.e. $L^\# = L$).

\begin{rem}
It is somewhat more customary for authors to define the dual lattice \cite[\S82F, p230]{OMeara:1971zr} as the {\bf Gram dual lattice},
$$
L_G^{\#} := \{\v \in V \mid B(\v, \w) \in R \text{ for all } \w \in L\}.
$$
and for the analogous notion of unimodular and self-dual to be {\bf Gram unimodular} (i.e. $\det(B) \in R^\times$) and  {\bf Gram self-dual} (i.e. $L_G^\# = L$).  
While either definition will suffice for a Jordan splitting theorem (Theorem \ref{Thm:Jordan_splitting}) in terms of unimodular lattices (because Gram and Hessian unimodular lattices are simply scaled versions of each other), the Hessian definitions are more natural from an arithmetic perspective (e.g. in the definition of level, the proof of Theorem \ref{Thm:Jordan_splitting}, and our discussion of neighboring lattices in Section \ref{Sec:Neighboring_lattices}).  If $2$ is invertible in $R$, then there is no distinction between the Hessian and Gram formulations, so over local ($p$-adic) rings this only makes a difference over the 2-adic integers $\Z_2$.
\end{rem}


\section{Quadratic Forms over Local ($p$-adic) Rings of Integers}

If we consider quadratic forms over the ring of integers $\Z_p$ of the $p$-adic field $\Q_p$, then the classification theorem is more involved because the valuation and units will both play a role. The main result along these lines involves the notion of a ``Jordan splitting'', which breaks $Q$ into a sum of pieces scaled by powers of $p$ which are as simple as possible. 



\begin{lem} \label{Lem:Splitting_lemma}
Suppose that $R$ is a principal ideal domain, $L$ is a quadratic $R$-lattice in the quadratic space $(V,Q)$, and $W$ is a non-degenerate subspace of $(V,Q)$.  If $\scale(L\cap W) = \scale(L)$, then we can write $L = (L\cap W) \oplus (L \cap W^{\perp})$.
\end{lem}

\begin{proof}
Since $R$ is a principal ideal domain, we know that all finitely generated $R$-modules are free \cite[Thrm 7.1, p146]{Lang:1995kx}, giving $L \equiv R^n$ and $L\cap W\equiv R^k$ where $V$ and $W$ have dimensions $n$ and $k$ respectively.  The structure theorem for finitely-generated $R$-modules \cite[Thrm 7.8(i), p153]{Lang:1995kx} says that 
%
%
we can find a set of $n$ vectors $\{\w_1, \dots, \w_k, \v_{k+1}, \dots \v_n\}$ that generate $L$ whose first $k$ elements generate $L\cap W$.  Without loss of generality we can scale the quadratic form $Q$ so that $\scale_H(L) = R$, in which case the Hessian matrix of $L\cap W$ in this basis is in $\GL_n(R)$.  Therefore for each generator $\v_j$ of $L$, we can realize the vector $(H(\w_i, \v_j))_{1\leq i \leq k}$ as an $R$-linear combination of its columns, and this linear combination can be used to adjust each $\v_j$ to be orthogonal to all $\w_i$, completing the proof.
%
%
%
%
\end{proof}

\begin{thm}[Jordan Decomposition] \label{Thm:Jordan_splitting}
A non-degenerate quadratic form over $\Z_{p}$ can be written as a direct sum 
$$
Q(\x) = \bigoplus_{j \in \Z} p^{j}Q_{j}(\x_{j})
$$
where the $Q_{j}(\x_{j})$ are unimodular.  
More explicitly, if $p>2$ then each $Q_{j}$ is a direct sum of quadratic forms $u_{i}x^{2}$ for some $p$-adic units $u_{i} \in \Z_p^\times$, and if $p=2$ then each $Q_{j}$ is a direct sum of some collection of the unimodular quadratic forms $u_{i}x^{2}$, $xy$, and $x^2 + xy + y^2$.
\end{thm}
\begin{proof}
This follows from successively applying Lemma \ref{Lem:Splitting_lemma}, and noticing that the minimal scale for a sublattice of $L$ can always be attained by a rank $1$ sublattice when $p\neq 2$, and be a rank $2$ sublattice when $p=2$.  The explicit statement for $p=2$ follows from checking equivalences between the rank 2 unimodular lattices.

Most authors state this result in terms of Gram unimodular lattices.
When $p>2$ this is given in Cassels's book \cite[Lemma 3.4 on p115]{Cassels:1978aa}, while $p=2$ is given in the very explicit form stated here as \cite[Lemma 4.1 on p117]{Cassels:1978aa}.  See also \cite[Thrm 8.1, p162 and Thrm 8.9, p168]{Gerstein:2008jh} and \cite[Thrm 93:29 on pp277]{OMeara:1971zr}.
For the general classification of integral quadratic forms over number fields at primes over $p=2$ see O'Meara's book \cite[Theorem 93:28 on pp267-276]{OMeara:1971zr}, though there only invariants (and not explicit representatives) are given.
%
\end{proof}

\begin{rem}
As a convention, we consider the ring of integers of $\R$ and $\C$ to be just $\R$ and $\C$ again, so there is nothing new to say in that situation.
\end{rem}



\section{Local-Global Results for Quadratic forms} \label{Sec:QF_local_global}

A useful idea for studying quadratic forms over either $\Q$ or $\Z$ is to consider them locally over all completions (by thinking of their coefficients in the associated local field $\Q_v$ or ring $\Z_v$), and then try to use information about these ``local'' quadratic forms to answer questions about the original ``global'' quadratic form.  While it is easy to pass from $Q$ to a local quadratic form $Q_v$ defined over its completion at the valuation $v$, it is more difficult to reverse this process to glue together a set of local forms $Q_v$ for all $v$ to obtain some ``global'' quadratic form $Q$.  

We now examine the extent to which this can be done.
Our first theorem tells us that for quadratic spaces over $Q$ this ``local-global'' procedure works flawlessly, and we can check the (rational) equivalence of forms using only local information.
%

\begin{thm}[Hasse-Minkowski Theorem]
Given two quadratic forms $Q_1$ and $Q_2$ defined over $\Q$, we have
$$
Q_{1} \sim_{\Q} Q_{2} \iff Q_{1} \sim_{\Q_{v}} Q_{2} \text{ for all places $v$ of $\Q$}.
$$
\end{thm}
\begin{proof}
This is stated as the ``Weak Hasse principle'' in Cassels's book \cite[Thrm 1.3 on p77]{Cassels:1978aa}, but proved in \cite[\S6.7, p85--86]{Cassels:1978aa}.
\end{proof}

\begin{rem}
The same result holds if we replace $\Q$ with any number field $K$, and replace the $\Q_v$ with all of the completions $K_v$ at all places of $K$.
\end{rem}



\medskip
We denote the $\Z$-equivalence class of $Q$ by $\Cls(Q)$, and refer to it as the {\bf class} of $Q$.
Given two quadratic forms $Q_1$ and $Q_2$ over $\Z$, we always have that
$$
Q_1 \sim_\Z Q_2 \implies Q_1 \sim_{\Z_v} Q_2 \quad \text{for all places $v$}
$$
since the linear transformation giving the $\Z$-equivalence is also defined over each completion $\Z_v$.  (Recall that $\Z_\infty := \R$ by convention.)  However unlike with quadratic forms over $\Q$, we are not guaranteed that local equivalence over all $\Z_v$ will ensure equivalence over $\Z$.
The number of distinct $\Z$-equivalence classes of quadratic forms that are locally $\Z_v$-equivalent to $Q$ at all places is called the {\bf class number} $h_Q$ of the quadratic form $Q$, and the set of all forms with the same localization as $Q$ is called the {\bf genus of $Q$}, so $h_Q = |\Gen(Q)|$.


It is a major result of Siegel from the reduction theory of (either definite or indefinite) quadratic forms over $\R$ that $h_Q < \infty$.  The class number of an indefinite quadratic form of dimension $n\geq 3$ is particularly simple to compute, and can be found in terms of a few local computations, but the class number of a definite form is considerably more complicated to understand exactly.

\begin{thm}
The class number $h_{Q}$ is finite.
\end{thm}
\begin{proof}
This follows from the reduction theory of quadratic forms, which shows that every class of quadratic forms over $\Z$ has some representative (of the same determinant) whose coefficients lie in a compact set.  This together with the discreteness of the (integer) coefficients of $Q$ gives that there are only finitely many classes of quadratic forms of bounded discriminant.  A proof can be found in \cite[Thrm 1.1, p128 and Lemma 3.1, p135]{Cassels:1978aa}
\end{proof}

\begin{rem}
It is also useful to discuss the {\bf proper class} of $Q$, denoted $\Cls^+(Q)$ which is the set of all $Q'\in\Cls(Q)$ where $Q'(\x) = Q(M\x)$ with $\det(M)=1$.  Since $\det(M)\in\{\pm1\}$, we see that there are at most two proper classes in a class, and so there are also finitely many proper classes in a given genus. The notion of proper classes is only meaningful when $n$ is even (because when $n$ is odd the $n \times n$ scalar matrix $M=-1_n$ has $\det(-1_n)= -1$, so $\Cls(Q) = \Cls^+(Q)$), and is important for formulating the connection between proper classes of binary quadratic forms and ideal classes in quadratic number fields.  This connection is discussed further in the Bhargava's notes \cite{BHARGAVA}.
\end{rem}

There is also a somewhat more geometrical notion of the {\bf class and genus of a quadratic lattice} $L \subset (V,Q)$, by considering the orbit of $L$ under the action of the rational or adelic orthogonal group.  In the language of quadratic forms, this says that two (free) quadratic lattices are in the same class or genus iff any associated quadratic forms (by choosing bases for the lattices) are in the same class or genus (respectively).  This gives rise to a class number $h(L)$ which is the number of classes in the genus of $L$ (and is again finite), and this agrees with the class number of the associated quadratic form when $L$ is free.
This notion of class and genus of a lattice is discussed in \cite[Defn 9.7, p180--181]{Gerstein:2008jh} and will be revisited in section \ref{Sec:Auto_forms_on_Orthog}.

\medskip
Interestingly, while indefinite forms appear more complicated on the surface, their arithmetic is usually {\it easier} to understand than that of definite forms, as can be seen from the following theorems.  The main idea is due to Eichler who discovered that the arithmetic of a certain simply connected algebraic group called the {\it spin group}, which is a double covering of $\mathrm{SO}(Q)$ and is very easy to understand via a property called ``strong approximation''.  This naturally leads to a notion of {\bf (proper) spinor equivalence}, and we call the orbit of $L$ under this equivalence the {\bf (proper) spinor genus} $\Spn^+(L)$ of $L$.  
We will discuss these notions briefly in section \ref{Sec:Spinor_equivalence}.   Some references for further reading about this topic are \cite[pp186-191]{Cassels:1978aa}, \cite[pp177-8, 192]{Shimura:2010vn}, \cite[pp315-321]{OMeara:1971zr}, \cite[\S7.4, pp427-433]{Platonov:1994ve}, \cite{Kneser:1966ly}.

\section{The Neighbor Method}  \label{Sec:Neighboring_lattices}

In this section we describe the method of neighboring lattices due to Kneser, which gives a useful construction for enumerating all classes in a given (spinor) genus of quadratic forms.  The idea is that one can perform explicit operations on a given quadratic lattice $L$ to produce different lattices that are obviously locally integrally equivalent to $L$.  By doing this carefully, one can find representatives of all classes in the genus $\Gen(L)$.


\begin{defn}
Given two integer-valued quadratic lattices $L, L'  \subset (V, Q)$ and some prime $p\in \N$, 
we say $L$ and $L'$ are {\bf $p$-neighbors} 
if $[L : L\cap L'] =  [L' : L\cap L'] = p$ and $H(L, L')\not\subseteq\Z$.  
\end{defn}
%


\subsection{Constructing $p$-neighbors}

Given a  quadratic lattice $L$ in a non-degenerate quadratic space $(V,Q)$, we now explain how to construct all of its $p$-neighboring lattices $L'$ explicitly in terms of certain vectors in $L$.  
%





\begin{thm}
If $p\in\N$ is prime, then every $p$-neighboring lattice $L'$ of a given $\Z$-valued primitive 
quadratic lattice $L$ with $\scale_H(L) = \Z$ has the form 
$$
L' =  \tfrac{1}{p}\w + L_{\w, p, \perp}
$$
where 
$$
L_{\w, p, \perp} := \{\v \in L \mid H(\v,\w) \equiv 0 \,\,\,(\mathrm{mod}\, p) \},
$$
for some primitive vector $\w \in L$ with $p^2\mid Q(\w)$.
\end{thm}

\begin{proof}
Any index $p$ superlattice $L'$ of $L''$ must be of the form $L' = L'' + \frac{1}{p}\w$ for some primitive vector $\w$ in $L''$, because by the structure theorem \cite[Thrm 7.8(i), p153]{Lang:1995kx} one can choose a basis for $L'$ starting with some $\w$ so that replacing $\w$ by $p\w$ gives a basis for $L''$.  For such an $L'$ to be $\Z$-valued we must at least have $Q(\frac{1}{p}\w) \in \Z$, which is equivalent to $p^2 \mid Q(\w)$.  Further since every $\x \in L'$ can be written as $\x = \y + \frac{a}{p} \w$ for some $\y \in L_{\w, p, \perp}$ and some $a\in \Z$, we have 
\begin{align*}
Q(\x) 
&= Q(\y + \tfrac{a}{p} \w) \\
&= Q(\y) + H(\y, \tfrac{a}{p} \w) + Q(\tfrac{a}{p} \w) \\
&=  \underbrace{Q(\y)}_{\in\Z} 
+  {\tfrac{a}{p}H(\y, \w)}
+ \underbrace{aQ(\tfrac{1}{p} \w)}_{\in\Z} 
\in \Z,
\end{align*}
and so we must have $H(\y, \w)\in p\Z$ for all $\y \in L''$.  However this condition defines an index $p$ sublattice of $L$, since it is the kernel of the surjective homomorphism $L \ra \Z/p\Z$ defined by $\v \mapsto H(\v, \w)$.  By reversing our reasoning, we see that all such $L'$ are $p$-neighbors of $L$.
\end{proof}

An important fact about $p$-neighbors $L'$ of $L$ is that they are all in the same genus $\Gen(L)$.  It is interesting to ask how many classes in the genus of $L$ can be created by taking repeated $p$-neighbors starting from $L$.  If one is allowed to vary the prime $p$, then this $p$-neighbor procedure gives all (proper) classes in $\Gen(Q)$.

\begin{thm}
\begin{enumerate}
\item If $L'$ is a $p$-neighbor of $L$ then $L'\in\Gen(L)$.  

\item If $p\nmid 2\det(L)$ and $n \geq 3$, then any $L'' \in \Spn^+(L)$ can be obtained by taking repeated $p$-neighbors of $L$. 

\item If the prime $p$ is allowed to vary, then we can obtain all proper classes $\Cls^+(L)$ in $\Gen(L)$ by taking repeated $p$-neighbors of $L$.
\end{enumerate}
\end{thm}
\begin{proof}
This definition of $p$-neighbor and local equivalence of $p$-neighbors is proved in  \cite[\S3.1, pp31-35]{Tornaria:2005sy}.  The spanning of the spinor genus is proved in \cite[Prop1, p339]{Benham:1983fb}, and the spanning of the genus is proved in \cite[Thrm 2, p340]{Benham:1983fb}.
\end{proof}
In fact, one can make more precise statements about exactly which spinor genera appear by taking $p$ neighbors because the $p$-neighboring operation can always be realized by an element of spinor norm $p(\Q^\times)^2$.  The image of this squareclass in the finite set of $\Q^\times$-squareclasses modulo spinor norms $\mathrm{sn}(O_\Q^+(V))$ and modulo the adelic spinor norms of the stabilizer of $O_\A^+(L)$ of $L$ determines exactly which of the (at most two) spinor genera can be reached by taking repeated $p$-neighbors of $L$.

%

\medskip
There is also a nice characterization of the $p$-neighbors of $L$ in terms of the non-singular points of the associated hypersurface $Q(\x) = 0$ over $\F_{p}$.
\begin{thm}
The $p$-neighbors of $L$ are in bijective correspondence with the non-singular points of $Q(\x) = 0$ in $\Proj^{n-1}(\F_{p})$.
\end{thm}
\begin{proof}
See \cite[Thrm 3.5, p34]{Tornaria:2005sy} or \cite[Prop 2.2, p739]{Scharlau:1998jo}.
\end{proof}

The $p$-neighbors can be organized into a weighted {\bf $p$-neighbor graph} whose vertices are the classes in $\Gen(L)$, where two vertices are connected by an edge iff they are $p$-neighbors, and where the multiplicity of each edge is the number of distinct $p$-neighboring lattices of $L$ which are equivalent $\sim_{\Z}$ to $L'$.  From the above theorem, we see that the $p$-neighbor graph is regular and that if $p\nmid 2\det(L)$ then the it is $p^{n-2}$-regular (i.e. every class has exactly $p^{n-2}$ neighbors for the prime $p$).

\chapter{Theta functions} \label{Chapter:ThetaFunctions}

\section{Definitions and convergence} 
We say that $m \in \Z$ is {\bf represented} by an integer-valued quadratic form $Q$ in $n$ variables if there is a solution $\x \in \Z^n$ to the equation $Q(\x) = m$.  Similarly we say that $m$ is {\bf locally represented} by $Q$ if 
there is a solution of $Q(\x) = m$ with $\x \in \R^n$ and also a solution $\x \in (\Z/M\Z)^n$ for every $M\in \N$.
%
Our main purpose 
in this chapter
will be to study the {\bf representation numbers}
$$
r_Q(m) := \#\{\x \in \Z^n \mid Q(\x) = m\}
$$
of a positive definite quadratic form $Q$ over $\Z$,
in order to understand
something about 
which numbers $m \in \Z \geq0$ are represented by $Q$.  
Our assumption here that $Q$ is positive definite ensures that $r_Q(m) < \infty$, since there are only finitely many lattice points (in $\Z^n$) in the compact solid ellipsoid $\mathcal{E}_m: Q(\x) \leq m$ when $\x \in \R^n$.  

It will also be important to consider the {\bf (integral) automorphism group of $Q$}, which is defined as
the set of invertible integral linear transformations preserving $Q$, i.e.
$$
\Aut(Q) := \{ M \in M_n(\Z) \mid Q(M\x) = Q(\x) \text{ for all } \x \in \Z^n\}.
$$
Our previous compactness observation  also tells us that $\#\Aut(Q) < \infty$, since any automorphism of $Q$ is determined by its action on a basis of $\Z^n$ and by taking $m$ large enough we can arrange that the (finitely many) integral vectors 
in $\mathcal{E}_m$ 
span $\Z^n$.  Because automorphisms preserve the values $Q(\x)$ of all vectors, they preserve the set of integral vectors inside the ellipsoid $\mathcal{E}_m$, and so there are only finitely many possible images of any specified spanning set.
%

In this setting it makes sense to define the {\bf theta series of $Q$} as the Fourier series generating function for the representation numbers $r_Q(m)$ given by
$$
\Theta_Q(z) := \sum_{m=0}^\infty r_Q(m) e^{2 \pi i m z}.
$$
From this perspective, our main goal will be to understand the symmetries of this generating function very well, and to use them to obtain information about the 
representation numbers $r_Q(m)$.

In order to make $\Theta_Q(z)$ more than just a formal object, we should try to establish a some convergence properties so it can be regarded as an honest function.  For this series to converge absolutely we need the exponentials in the sum to be decaying, which happens for $z\in\C$ when $\Im(z)>0$.  For convenience, we denote by $\H$ the complex upper half-plane
$$
\H := \{ z \in \C \mid \Im(z) > 0\}.
$$
The following theorem shows that any Fourier series with moderately (i.e. polynomially) growing coefficients will converge absolutely on $\H$.
\begin{lem}[Convergence of Fourier series]
The Fourier series 
$$
f(z) := \sum_{m=0}^\infty a(m) e^{2 \pi i m z}.
$$
converges absolutely and uniformly on compact subsets of $\H$ to a holomorphic function $f:\H\ra \C$ if all of its coefficients $a(m)\in \C$ satisfy
$|a(m)| \leq Cm^{r}$ for some constant $C>0$.
\end{lem}
\begin{proof}
See \cite[Lemma 4.3.3, p117]{Miyake:2006hf}.
\end{proof}

\noindent
Because the number of lattice points in a smooth bounded region $\mathcal{R} \subset \R^n$ is approximately $\vol(\mathcal{R})$, we see that $\sum_{i=0}^M r_Q(m) < C M^n$ for some constant $C$.  Therefore for each $m$ individually we must have that $r_Q(m) < C_1 M^{n-1}$ for some constant $C_1$,  so the previous lemma shows that the theta function $\Theta_Q(z)$ converges (absolutely and uniformly) to a holomorphic function when $z \in \C$ and $\Im(z) > 0$.   This gives the following important result:

\begin{thm}
The theta series $\Theta_Q(z)$ of a positive definite integer-valued quadratic form $Q$ converges absolutely and uniformly to a holomorphic function $\H \ra \C$.
\end{thm}

\section{Symmetries of the theta function}

While it is not obvious at first glance, $\Theta_Q(z)$ has a surprisingly large number of symmetries.  From its definition as a Fourier series, it is clearly invariant under the transformation $z \mapsto z+1$, but this is not particularly special since this holds for any Fourier series.  However there is
an additional symmetry provided to us by Fourier analysis because we can also view the theta function as a sum of a quadratic exponential function over a lattice $\Z^n$, i.e. 
$$
\Theta_Q(z) = \sum_{m=0}^\infty r_Q(m) e^{2 \pi i m z} = \sum_{\x \in \Z^n} e^{2 \pi i Q(\x) z}.
$$
This additional lattice symmetry is realized through the Poisson summation theorem:
\begin{thm}[Poisson Summation Formula]
Suppose that $f(\x)$ is a function on $\R^{n}$ which decays faster than any polynomial as $|\x| \ra \infty$ (i.e. for all $r\geq0$ we know that $|\x|^{r} f(\x) \ra 0$ as $|\x| \ra \infty$).  Then
the equality 
$$
\sum_{\x\in \Z^{n}} f(\x) = \sum_{\x\in \Z^{n}} \hat{f}(\x)
$$
holds and the sums on both sides are absolutely convergent, where 
$$
\hat{f}(\x) := \int_{\y\in \R^{n}} f(\y) e^{-2\pi i \x\cdot\y} \, dy
$$
is the Fourier transform of $f(\x)$.
\end{thm}
\begin{proof}
See \cite[pp249-250]{Lang:1994fu} for a proof of this.
\end{proof}

The important point here is that the Gaussian function $f(x) = e^{-\pi \al x^2}$ transforms into a multiple of itself under the Fourier transform (which follows essentially from checking that $e^{-\pi x^2}$ is its own Fourier transform).  Writing $z=x+iy \in \H$ in $\Theta_Q(z)$, we see that the $y$-dependence of each term will look like a decaying Gaussian (while the $x$-dependence will just oscillate), so Poisson summation allows us to transform each term into itself after a little rescaling.  This allows us to establish an additional symmetry for the theta function under the transformation $z\mapsto \frac{-1}{Nz}$ for some $N\in\N$.  In the special case where $Q(\x) = x^2$ we can take $N=4$ and have the two identities  
\begin{align}
\Theta_{x^2}(-1/4z) = \sqrt{-2iz} \,\, \Theta_{x^2}(z)
\qquad\text{and}\qquad 
\Theta_{x^2}(z+1) = \Theta_{x^2}(z).
\end{align}
By extending these to the group generated by the transformations $z \mapsto \frac{-1}{4z}$ and $z \mapsto z+1$, we obtain following prototypical theorem.

\begin{thm}  \label{Thrm:Theta_one_var}
For all $\[\begin{smallmatrix}a & b \\ c & d\end{smallmatrix}\] \in \SL_2(\Z)$ with $4 \mid c$, we have that 
$$
\Theta_{x^{2}}\(\frac{az+b}{cz+d}\) = \ve_d^{-1}\(\frac{c}{d}\) \sqrt{cz+d} \,\, \Theta_{x^{2}}(z)
$$
where $\frac{-\pi}{2} < \arg(\sqrt{z}) \leq \frac{\pi}{2}$, 
$$
\ve_d := 
\begin{cases}
1 & \text{if $d \equiv 1 \mod 4$,} \\
i & \text{if $d \equiv 3 \mod 4$,} \\
\end{cases}
\qquad
and 
\qquad
\leg{c}{d} :=
\begin{cases}
\leg{c}{|d|}  & \text{if $c>0$ or $d>0$,} \\
-\leg{c}{|d|}  & \text{if both $c, d < 0$.} \\
\end{cases}
$$
Here when $d>0$ the symbol $\leg{c}{d}$ agrees with the usual quadratic character mod $d$.

\end{thm}
\begin{proof}
This is stated in \cite[(1.10), p440]{Shimura:1973aa} and proved in Prop 2, p457.  See also Iwaniec \cite[Thrm 10.10 with $\chi = \mathbf{1}$, pp177-8]{Iwaniec:1997ph}, Knopp \cite[Thrm 13, p46 and Thrm 3, p51]{Knopp:1970ul}, \cite[Cor 4.9.7, p194]{Miyake:2006hf} and Andrianov-Zhuralev \cite[Prop 4.15, p42]{Andrianov:1995kc} for proofs. 
\end{proof}

For a diagonal quadratic form $Q(\x) = \sum_{i=1}^n a_i x_i^2$ of level $N$, this formula is enough to see that $\Theta_Q(z)$ transforms into a multiple of itself under the element $z\mapsto  \frac{-1}{Nz}$.  However to obtain a transformation formula for a general theta series $\Theta_Q(z)$ similar to Theorem \ref{Thrm:Theta_one_var}, a more general strategy is needed.  
One approach is to compute the transformation formulae for more general theta series involving a linear term, and then to specialize this term to zero.
%
Another approach is to
obtain identities for how a related generalized higher dimensional theta function (similar to $\Theta_{x^2}(z)$) transforms with respect to a special subgroup of $\Sp_{2n}(\Z)$ called the {\bf theta group}.   
In the case where $n=1$, the theta group consists the elements $\[\begin{smallmatrix}a & b \\ c & d\end{smallmatrix}\]$ of $\Sp_2(\Z) = \SL_2(\Z)$ where both products $ab$ and $cd$ are even.  Either approach allows one to show the following important transformation formula:



\begin{thm}
Suppose $Q$ is a non-degenerate positive definite quadratic form over $\Z$ in $n$ variables with level $N$.  Then for all $\[\begin{smallmatrix}a & b \\ c & d\end{smallmatrix}\] \in \SL_2(\Z)$ with $N \mid c$, we have that 
$$
\Theta_{Q}\(\frac{az+b}{cz+d}\) = \(\frac{\det(Q)}{d}\) \[ \ve_d^{-1}\leg{c}{d} \sqrt{cz+d}\]^n \,\, \Theta_{Q}(z),
$$
where $\sqrt{z}$, $\ve_d$, and $\leg{c}{d}$ are defined in Theorem \ref{Thrm:Theta_one_var}.
\end{thm}

\begin{proof}
A nice discussion of theta series and their transformation formulas (by the first approach) can be found in \cite[Ch. 10]{Iwaniec:1997ph}, and a somewhat simpler discussion of transformation formulas for theta series in an even number of variables along these lines is given in the appendix to Chapter 1 of Eichler's book \cite[pp44-52]{Eichler:1966rp}.  
The second approach described above can be found (in much greater generality) in \cite[Ch1, \S3-4, pp11-37]{Andrianov:1995kc}, especially Thrm 3.13 on p22.
\end{proof}

Here the $2\times 2$ matrices that give symmetries of $\Theta_Q(z)$ form a subgroup of $\SL_{2}(\Z)$ called the {\bf level $N$ congruence group}, which is usually denoted as 
$$
\Gamma_{0}(N) := \left\{\begin{bmatrix} a & b \\ c & d \end{bmatrix} \in \SL_{2}(\Z) \biggm | c \equiv 0 \pmod{N}\right\}.
$$

%
%
%
%

\section{Modular Forms}

It is useful to understand theta series in the context of all functions that have symmetries with respect to the action of congruence subgroups $\Gamma_0(N)$ by linear fractional transformations on $\H$.
%
%
%
%
%
%
This idea leads us to 
define a very important class of functions called {\bf modular forms}, whose symmetry properties essentially depend on three parameters: the {\bf weight} $k \in \frac12 \Z$, the {\bf level} $N \in \N$, and the {\bf character} $\chi:(\Z/N\Z)^\times \ra \C^\times$.  If the weight $k \notin \Z$, then we must specify an additional function $\ve := \ve(\gamma, k)$ called a {\bf multiplier system}.  For theta series this is called the ``theta multiplier'', but we will not be concerned with its exact form here.

\begin{defn} \label{defn:modular_form}
We define a {\bf modular form} of weight $k$, level $N$, Dirichlet character $\chi$ and multiplier system $\ve$ to be a holomorphic function $f:\H \ra \C$ which transforms with respect to 
$\Gamma_{0}(N)$ 
under the rule 
\begin{equation} \label{Eq:modular_form_tranformation_law}
f\(\frac{az+b}{cz+d}\) =\ve(\gamma, k) \chi(d) (cz+d)^{k} f(z)  
\end{equation}
for all $\gamma:= \begin{bmatrix} a & b \\ c & d \end{bmatrix} \in \Gamma_{0}(N)$, and satisfies the additional technical condition that $f(z)$ is also ``holomorphic'' at the boundary values $\Proj^{1}(\Q) := \Q \cup \{\infty\}$ of the quotient $\Gamma_{0}(N) \bs \H$.  
\end{defn}

\noindent
It is standard notation to let $M_k(N,\chi)$ denote the $\C$-vector space of all modular forms of weight $k \in \frac{1}{2}\Z$, level $N$ and character $\chi$, 
where we assume the the {\bf trivial multiplier system} $\ve(\gamma, k) := 1$ if $k \in \Z$ and the {\bf theta multiplier system} $\ve(\gamma, k) := \ve_d^{-1}\leg{c}{d}$ if $k \not\in \Z$.

\medskip


%

We can now rephrase the symmetries of the theta function $\Theta_Q(z)$ using the language of modular forms.  Good references for the general theory of modular forms are \cite{Iwaniec:1997ph, Diamond:2005mw, Miyake:2006hf, Koblitz:1993vc, Shimura:1994ab}, and the theta multiplier is described in detail in \cite[Ch 4]{Knopp:1970ul} and \cite[Ch 10]{Iwaniec:1997ph}.
One important observation to make about modular forms $f(z)$ is that the element $\[\begin{smallmatrix} 1 & 1 \\ 0 & 1 \end{smallmatrix}\] \in \Gamma_0(N)$, and so the transformation formula above shows that $f(z+1) = f(z)$.
\footnote{To justify this, notice that both the trivial and theta multiplier systems have value $1$ on this element.}
This periodicity together with the holomorphy of $f(z)$ shows that any modular form can be written as a complex {\bf Fourier series}
\begin{equation} \label{Eq:Fourier_expansion}
f(z) = \sum_{m=0}^\infty a(m) e^{2\pi i mz} = \sum_{m=0}^\infty a(m) q^m 
\qquad \text{where $q := e^{2\pi i z}$ }
\end{equation}
and the {\bf Fourier coefficients} $a(m) \in \C$.


\begin{cor}
Suppose $Q$ is a non-degenerate positive definite quadratic form over $\Z$ in $n$ variables with level $N$.  Then $\Theta_Q(z) \in M_{\frac{n}{2}}(N, \chi)$ is a modular form 
of weight $\frac{n}{2}$, level $N$ and character $\chi(\cdot) = \(\frac{(-1)^{\lfloor\frac{n}{2}\rfloor}\det(Q)}{\cdot}\)$ (and multiplier system $\ve(\gamma, k)$ specified above).
\end{cor}


\begin{proof}
This follows because $\ve_d^2 = \leg{-1}{d}$, and so 
$
\[ \ve_d^{-1} \leg{c}{d} \]^n = \leg{-1}{d}^{\lfloor\frac{n}{2}\rfloor} 
\cdot
\begin{cases}
\ve_d^{-1} \leg{c}{d} & \text{ if $n$ is odd,} \\
1 & \text{ if $n$ is even.} \\
\end{cases}
$
\end{proof}

\begin{rem}
Note that here the ``level'' $N$ refers both the level of the quadratic form as well as the level of the modular form (i.e. we use symmetries from $\Gamma_0(N)$).
\end{rem}

To understand modular forms structurally, it is important to understand the action of $\Gamma_0(N)$ on $\H$ by linear fractional transformations $z \mapsto \frac{az+b}{cz+d}$.  
When $N=1$, then $\Gamma_0(N)$ is all of $\SL_2(\Z)$ and there is a well-known fundamental domain $\mathcal{F}$ for this action given by 
$$
\mathcal{F} := \{z \in \H \mid |z| \geq 1 \text{ and } |\Re(z)| \leq \tfrac12\},
$$
together with some identifications of its boundary.  After these identifications have been made, the resulting fundamental domain $\mathcal{F}$ is not compact.   However $\mathcal{F}$ {\it can be  naturally extended to a compact surface} by adding one point (usually called $\infty$ or $i\infty$) which we imagine to be at the topmost end of the $y$-axis.  This point is called a {\bf cusp of $\SL_2(\Z)$} due to the apparent pointyness of $\mathcal{F}$ as we move along the $y$-axis towards $i\infty$.
In general, $\Gamma_0(N)$ has finite index in $\SL_2(\Z)$ and so its fundamental domain will be a union of finitely many translates of $\mathcal{F}$ (with slightly different boundary identifications).  This larger fundamental domain is again not compact, but here it can be made compact by the addition of finitely many ``boundary'' points which we call {\bf cusps of $\Gamma_0(N)$}.  These cusps can always be represented by elements of $\Proj^1(\Q)$ since they will be the image of the cusp $i\infty$ under some element of $\Gamma_0(N) \subset \SL_2(\Q)$, and we have the identification $i\infty = \infty \in \Proj^1(\Q) \subset \Proj^1(\C)$.

These cusps play a very important role in the theory of modular forms.  For example, they can be used to define a natural subspace of modular forms which vanish at all cusps, called the {\bf cusp forms} $S_k(N, \chi) \subseteq M_k(N, \chi)$.  Also, for each cusp $\mathcal{C}$ of $\Gamma_0(N)$ we can usually construct a special modular form $E_{\mathcal{C}}(z)$ associated to $\mathcal{C}$ which has value 1 at $\mathcal{C}$ and vanishes at all other cusps.  We call the space spanned by all of these functions associated to cusps the space of {\bf Eisenstein series} $E_k(N, \chi) \subseteq M_k(N, \chi)$.  

The Eisenstein series associated to cusps can be understood very explicitly, and is usually considered to be the ``easier'' part of $M_k(N, \chi)$.  For example, for the cusp $\mathcal{C} = i\infty$ of $\SL_2(\Z)$, the associated Eisenstein series $E_{\mathcal{C}}(z)$ of weight $k \in 2\Z >2$ is given by
\begin{equation} \label{Eq:Eis_series}
G_k(z) 
:= \tfrac{1}{2} \sum_{\substack{(c, d) \in \Z^2 \\ \gcd(c,d) = 1}}  \frac{1}{(cz+d)^k}
= 1 - \frac{2k}{B_{2k}} \sum_{m \geq 1}  \sigma_{k-1}(m) q^m  
\in M_k(N=1, \chi = \mathbf{1})
\end{equation}
where $B_{2k}$ is the $(2k)^\text{th}$ Bernoulli number, $\sigma_{k-1}(m) := \sum_{0 < d \mid m} d^{k-1}$ is the usual divisor function and $q := e^{2\pi i z}$. (See \cite[Lemma 4.1.6, p100 and Thrm 3.2.3, p90]{Miyake:2006hf}.)  We can also interpret the Fourier expansion in (\ref{Eq:Fourier_expansion}) as being associated with the cusp $i\infty$, since we can view $q$ as a local parameter in the neighborhood of $i\infty$.

\medskip
{\bf Facts about Modular forms:}
%
We now state several fundamental structural results in the theory of modular forms that are useful for understanding theta series:
\begin{enumerate}
\item The space of modular forms $M_{k}(N, \chi)$ with fixed invariants $(k, N, \chi)$ is a finite dimensional vector space over $\C$. \cite[\S2.5, pp57-61]{Miyake:2006hf}
\item The space $M_{k}(N, \chi)$ can be decomposed uniquely as a direct sum of cusp forms (of functions vanishing at all cusps) and Eisenstein series (spanned by the Eisenstein series associated to the cusps of $\Gamma_{0}(N)$). \cite[Thrm 2.1.7, p44 and Thrm 4.7.2, p179]{Miyake:2006hf}
\item Any Eisenstein series has Fourier coefficients $a_{E}(m)$ which can be as large as $c_\ve m^{k-1 + \ve}$ for any $\ve > 0$ and some constant $c_\ve\in \R > 0$. 
\cite[Thrm 4.7.3, p181]{Miyake:2006hf}
\item Any cusp form has Fourier coefficients $a_{f}(m)$ which are (trivially) no larger than $c_\ve m^{\frac{k}{2} + \ve}$ for any $\ve >0$ and some constant $c_\ve\in \R > 0$.
\cite[Cor 2.1.6, p43]{Miyake:2006hf}
\end{enumerate}

For our purposes, it is important to note that the upper bound on Eisenstein coefficients is not far from the truth, and is best possible when $k >2$.  When $k \in 2\Z > 2$, this bound is attained by the Eisenstein series (\ref{Eq:Eis_series}).

\section{Asymptotic Statements about $r_Q(m)$}
To apply our knowledge of modular forms to study the numbers $m$ represented by $Q$, we write the theta series as 
$$
\Theta_Q(z) = E(z) + C(z)
$$ 
where $E(z)$ is an Eisenstein series and $C(z)$ is a cusp form.  Looking at the $m^\text{th}$ Fourier coefficient  of this equation gives a decomposition of the representation numbers as  
$$
r_Q(m) = a_E(m) + a_C(m).
$$

From our informal discussion of modular forms above we know that the Eisenstein Fourier coefficients $a_E(m)$ are about as large as $m^{k-1}$ as $m\ra\infty$, and when $n=5$ one can show using (\ref{Eq:Siegel_product})  that 
$$
|a_E(m)| >> m^{\frac{3}{2}}
$$
when they are non-zero, and when $n \geq 4$ this occurs $\iff m$ is locally represented by $Q$.   Similarly we know that the cusp form Fourier coefficients satisfy
$$
|a_C(m)| << m^{\frac54 + \ve},
$$
so if the Eisenstein coefficients are non-zero, then we know that $r_Q(m)$ is non-zero and so $m$ is represented by $Q$ if $m$ is sufficiently large.  This asymptotic estimate only improves when $Q$ has more variables, giving the following theorem originally due to Tartakowski:

\begin{thm}[Tartakowski, \cite{Tartakowsky:1929dq}]
If $Q$ is a positive definite quadratic form over $\Z$ in $n \geq 5$ variables, then every sufficiently large number $m \in \N$ that is locally represented by $Q$ is represented by $Q$.
\end{thm}

%
%
%

For $n\leq 4$, the above results are not enough to show that the Eisenstein coefficients are asymptotically larger than the cusp form coefficients, so more care is needed.  The case $n=4$ was first handled by Kloosterman by a clever refinement of the Circle Method (described briefly below), and has since been absorbed into the theory of modular forms as a consequence of the Ramanujan bound $|a_{f}(p)| \leq 2\sqrt{p}$ for prime coefficients of weight 2 cusp forms.  This case also involves additional local considerations at finitely many primes $p$ where  $Q$ is anisotropic over $\Q_{p}$.

The case $n=3$ is even more delicate, and involves additional arithmetic and analytic tools to understand (e.g. spinor genera, the Shimura lifting of half-integral weight forms, analytic bounds on square-free coefficients of half-integral weight forms).  In particular it was handled by Duke and Schulze-Pillot, and then by Schulze-Pillot in the papers \cite{Duke:1990ay, Schulze-Pillot:2000fm}  For more details on asymptotic results, see the survey papers \cite{Hanke:2004xa, Duke:1997ko, Iwaniec:1987tt, Schulze-Pillot:2004ir}.  

The case $n=2$ of binary forms is a genuinely arithmetic problem (since for weight $k=1$ both cusp forms and Eisenstein series coefficients 
satisfy $a(m) << m^\ve$ for any $\ve > 0$ \cite[\S5.2(c), p220]{Serre:1977cr})
and it exhibits a much closer connection to explicit class field theory for quadratic extensions than the asymptotic results described here.

\section{The circle method and Siegel's Formula}

The origins of the many of the modern analytic techniques in the theory of quadratic forms have their origins in the famous ``circle method'' of Hardy, Littlewood and Ramanujan.  The idea of this method is that one can express the number of representations $r_{Q}(m)$ for $Q = a_{1}x_{1}^{2} + \cdots + a_{n}x_{n}^{2}$ as an integral over the unit circle which can be well-approximated by taking small intervals about angles which are rational multiples of $2\pi$ (where small here means small relative to the overall denominator of the rational multiples one considers).  These rational angle contributions can be thought of locally (in terms of Gauss sums), and so we learn that local considerations give a good approximation of the number of representations $r_{Q}(m)$ for $Q = a_{1}x_{1}^{2} + \cdots + a_{n}x_{n}^{2}$ when $n$ is large enough.  In the language of modular forms this method produces an Eisenstein series $E_{Q}(z)$ (called a ``singular series'') with multiplicative Fourier coefficients that agrees with the theta series $\theta_{Q}(z)$ at all rational points (and at $\infty$) so the difference $\theta_{Q}(z) -E_{Q}(z)$ is a cusp form and so must have asymptotically smaller Fourier coefficients than $E(z)$.  This cusp form can be analyzed to various degrees, but the most naive bound for its Fourier coefficients gives non-trivial asymptotic information for the asymptotic behavior of $r_{Q}(m)$ for $m\geq 5$.  (See \cite[Ch 6, pp ???]{Moreno:2006qf} and \cite[Ch 5, pp63--87]{Knopp:1970ul} for details.)  The case $n=4$ can also be handled, but requires an essential refinement of Kloosterman to obtain additional cancellation.  (See \cite[\S11.4--5, pp190--199]{Iwaniec:1997ph} and \cite[\S20.3-5, pp467--486]{Iwaniec:2004la} for more details.)

Siegel used these ideas to give quantitative meaning to the Fourier coefficients in the singular series both in terms of the underlying space of modular forms (as an Eisenstein series), but also in terms of the ``local densities'' associated to the quadratic form $Q$.  In particular he proved the following theorem:

\begin{thm}[Siegel]  \label{Thm:Siegel_formulas}
Suppose $Q(\x)$ is a positive definite integer-valued quadratic form in $n\geq 5$ variables, whose theta series $\Theta_{Q}(z)$ is written (uniquely) as a sum of an Eisenstein series $E(z)$ and a cusp form $C(z)$.  Then the Eisenstein series 
$$E(z) = \sum_{m\geq 0} a_{E}(m) e^{2 \pi i m z}$$ 
can be expressed in two different ways:

Firstly, $E(z)$ can be recovered as a weighted sum of theta series over the genus of $Q$:
\begin{equation} \label{Eq:Siegel_avg}
E(z) = \frac{\sum_{Q'\in \Gen(Q)} \frac{\Theta_{Q'}(z)}{|\Aut(Q')|}}{\sum_{Q'\in \Gen(Q)} \frac{1}{|\Aut(Q')|}},
\end{equation}
showing that $E(z)$ is a genus invariant. (I.e., The theta series of any $Q' \in \Gen(Q)$ will have the same Eisenstein series $E(z)$.)

Secondly, the Fourier coefficients $a_{E}(m)$ can be expressed as an infinite local product
\begin{equation} \label{Eq:Siegel_product}
a_{E}(m) = \prod_{\text{places}\,  v} \beta_{Q,v}(m)
\end{equation}
where the numbers $\beta_{Q,v}(m)$ are the {\bf local representation densities of $Q$ at $m$}, defined by the limit
\begin{equation} \label{Eq:Local_density_defn}
\beta_{Q,v}(m) := 
\lim_{U \ra \{m\}}
\frac{\vol_{\Z_v^n}({Q^{-1}(U)})}{\vol_{\Z_v}(U)}
\end{equation}
where $U$ runs over a sequence of open subsets of $\Z_v$ containing $m$ with common intersection $\{m\}$, and the volumes appearing are
the natural translation-invariant volumes
on $n$-dimensional and 1-dimensional affine space over $\Z_v$ of total volume one.
\end{thm}
\begin{proof}
See Siegel's Lecture notes \cite{Siegel:1963yt} or his original series of papers \cite{Siegel:1935qf, Siegel:1936qo, Siegel:1937jw}.
\end{proof}
These formulas are extremely important for the analytical theory of quadratic forms, and can be used to provide precise asymptotics for $r_{Q}(m)$ as $m\ra \infty$.
Extensions of this technique led Siegel to prove analogous results for more general kinds of theta functions which count representations of a quadratic form by another quadratic form.  These are examples of  ``Siegel modular forms'' which have analogous symmetries for the symplectic group $\Sp_{2r}$. (See \cite{Andrianov:1995kc} for more details.)

The formulas of Siegel were later generalized by Weil to a more representation-theoretic context by means of a certain very simple representation of a symplectic group called the ``Weil representation'' that we will meet later.  This representation can be used to give a proof of Siegel's formulas in the case where $Q$ is a positive definite quadratic form in $n\geq 5$ variables, and has been extended by Kudla and Rallis \cite{Kudla:1988aa, Kudla:1988ab} to cover many more cases, including $n \geq 3$.  It is interesting to see the progression of ideas from the circle method to modular forms to the Weil representation, and to notice that while the language used to obtain these results changes to suit our deepening perspective and context, the essential features (and technical difficulties) of the result remain very much the same.  

These structural results about theta series and modular forms can also be generalized to understand theta series of totally definite $\O_F$-valued quadratic forms over totally real number fields $F$.
These theta series are
then
Hilbert modular forms for a congruence subgroup of the group $\GL_2(\O_F)$ where $\O_F$ is the ring of integers of $F$.  They can also be generalized to understand the number of representations of a smaller quadratic form $Q'$ by $Q$, where this can be viewed in the lattice picture as counting the number of isometric embeddings of the quadratic lattice $L'$ into $L$ (which are quadratic lattices associated to $Q'$ and $Q$ respectively).  In this context, the resulting theta series is a Siegel modular form for some congruence subgroup of the symplectic group $\Sp_{2n'}(\Z)$, where $Q'$ is a quadratic form in $n'$ variables.  (Notice that in the special case where $n' = 1$ we have $\Sp_2 = \SL_2$.)  In both of these settings, Siegel's formulas remain essentially unchanged.

%
%
%
%

\section{Mass Formulas}

One useful application of the generalizations of Siegel's formula to representing quadratic forms $Q'$ by a quadratic form $Q$  is when we take $Q' = Q$.  In this case, a generalization of Siegel's first formula (\ref{Eq:Siegel_avg}) applied to the $Q^\text{th}$-Fourier coefficient of the associated Siegel modular form gives
\begin{align}
a_E(Q) 
= \frac{\sum_{Q''\in \Gen(Q)} \frac{r_{Q''}(Q)}{|\Aut(Q'')|}}{\sum_{Q''\in \Gen(Q)} \frac{1}{|\Aut(Q'')|}}
= \frac{1}{\sum_{Q''\in \Gen(Q)} \frac{1}{|\Aut(Q'')|}}
\end{align}
because the number of representations $r_Q(Q'')$ of any quadratic form $Q''\in \Gen(Q)$ by $Q$ is given by
$$
r_Q(Q'') = 
\begin{cases}
\#\Aut(Q) & \quad \text{if $Q'' \sim_\Z Q$,} \\
0 & \quad \text{if $Q'' \not\sim_\Z Q$.}
\end{cases}
$$
From an extension of Siegel's second formula (\ref{Eq:Siegel_product}), we also see that $a_E(Q)$ can be written as a product of local densities (though in this case an extra factor of 2 is needed).  This motivates the definition of the {\bf mass of a quadratic form $Q$}, denoted by $\Mass(Q)$, as 
$$
\Mass(Q) := \sum_{Q''\in \Gen(Q)} \frac{1}{|\Aut(Q'')|}.
$$
By Siegel's theorems we see that the mass is a local quantity, and can be computed from local knowledge about $Q$ over $\Z_v$ at all places $v$.  

Explicit computations of the mass are simple in principle, but often a bit painful to make explicit.  These are known as ``exact mass formulas'', and they provide very useful information about the class number $h_Q$ of a genus $\Gen(Q)$.  As an example of this, using the fact that every quadratic form has at least two automorphisms (e.g. $\x \mapsto \pm \x$) we can see that 
$$
\Mass(Q) 
= \sum_{Q''\in \Gen(Q)} \frac{1}{|\Aut(Q'')|} 
\leq \frac{h(Q)}{2}.
$$
Therefore if the $\Mass(Q)$ is large then we know that the genus must contain many distinct classes.  However the size of the mass of a positive definite form can be shown by local considerations to grow as we vary $Q$ in an infinite family (e.g. if $n$ grows, or if $\det(Q)$ grows and $n\geq 2$), so the class number can also be shown to get very large in these situations.  One interesting application of this is the following result of Pfeuffer  and Watson:
\begin{thm}
There are only finitely many (classes of) primitive positive definite quadratic forms $Q$ over $\Z$ in $n\geq 2$ variables with class number $h_Q=1$.
\end{thm}

\noindent
In a long series of papers \cite{Watson:1963aa}--\cite{Watson:1984aa}, Watson enumerated many of these class number one forms.
More generally, Pfeuffer showed that there are finitely many totally definite primitive integer-valued quadratic forms $Q$ in $n \geq 2$ variables with $h_Q = 1$  as we vary over all totally real number fields.  (See \cite{Pfeuffer:1971pd, Pfeuffer:1977ve} for details.)

%
%
%
%
%
%
%

\medskip 
It should also be noted that this is not the end of the story for mass formulas.  There are many other connections (e.g. to Tamagawa numbers, Eisenstein series on orthogonal groups, and computational enumeration of classes in a genus) that we do not have space to mention here.  As an example of one continuation of the story, in the past few years Shimura has defined a somewhat different notion of ``mass'' and ``mass formula'' for quadratic forms 
which instead of dividing the number of representations by the number of automorphisms, it
counts the number of {\it equivalence classes} of representations 
 in a genus {\it under the action of the automorphism group}.  For a nice discussion of these see \cite{Shimura:2006tg, Shimura:2006ai} and \cite[\S37]{Shimura:2010vn}, as well as the more detailed \cite[\S12--13]{Shimura:2004qe}. These are very interesting, but do not fit within the framework we are describing here.  They are also a good example of how a well-established theory is still evolving in new ways, and that there are many avenues left for future researchers to explore!

\section{An Example: The sum of 4 squares} \label{Sec:local_densities_for_4_squares}

We conclude with a concrete example of how Siegel's formulas can be used to understand how many ways we can represent certain numbers as a sum of four squares.  This question can be treated in many different ways, but the most definitive result is the following exact formula of Jacobi which he derived via the theory of elliptic functions.

\begin{thm}[Jacobi \cite{Jacobi.:1829hc}]  \label{Thm:Sum_of_four_squares}
For $m \in \N$, we have 
$$
r_{x^2 + y^2 + z^2 + w^2}(m) = 8 \cdot \hspace{-0.1in}\sum_{\substack{0 < d \mid m \\ 4\nmid d}} d.
$$
\end{thm}

We will now derive some special cases of this result for certain $m$ by using Siegel's formulas in Theorem \ref{Thm:Siegel_formulas}.  To do this, we first note that $Q(\x) = x^2 + y^2 + z^2 + w^2$ has class number $h_Q = 1$, so Siegel's first formula gives
$$
a_E(m) 
= \frac{\sum_{Q''\in \Gen(Q)} \frac{r_{Q''}(m)}{|\Aut(Q'')|}}{\sum_{Q''\in \Gen(Q)} \frac{1}{|\Aut(Q'')|}}
= \frac{\frac{r_{Q}(m)}{\cancel{|\Aut(Q)|}}}{\frac{1}{\cancel{|\Aut(Q)|}}}
= r_Q(m).
$$
Now we can apply Siegel's second formula to give the purely local formula
$$
r_Q(m) = a_E(m) = \prod_v \beta_{Q,v}(m)
$$
for $r_Q(m)$ in terms of local densities $\beta_{Q,v}(m)$ defined in (\ref{Eq:Local_density_defn}).  We now compute this infinite product to evaluate $r_Q(m)$ for some $m \in \N$.  For convenience of notation, from now on we use the abbreviation $\beta_{v}(m) := \beta_{x^2 + y^2 + z^2 + w^2\!,\, v}(m)$.

\subsection{Canonical measures for local densities}
%
%
%
To compute the local densities $\beta_v(m)$ defined in (\ref{Eq:Local_density_defn}) we use the
``canonical'' Haar measures $\mu$ on $\Z_v$ (i.e. additively invariant) uniquely defined by the normalizations
$$
\mu_{\Z_p}(\Z_p) = 1,
\hspace{.3in}
\mu_{\R}([0,1]) = 1.
$$
Even if one is unfamiliar with the measure $\mu_{\Z_p}$, the important thing is that we can easily compute the measure of any set we are interested in.  In particular, because we can write $\Z_p$ as the disjoint union 
$$
\Z_p = \bigsqcup_{a\in \Z/p^i\Z} a + p^i\Z_p
$$
and each of these cosets has the same measure (by the additive invariance), we see that $\mu_{\Z_p}(p^i\Z_p) = \frac{1}{p^i}$ and also $\mu_{\Z_p^n}(p^i\Z_p^n) = \frac{1}{p^{n\cdot i}}$.

%
%

\subsection{Computing $\beta_\infty(m)$}

When $v=\infty$ we have $\Z_v = \R$, so we see that the local density $\beta_\infty(m)$ is the volume of a thin ``shell'' around the ellipsoid $x^2 + y^2 + z^2 + w^2 = m$ divided by the ``thickness'' of the shell (in $m$-space), which is some measure of the ``surface area'' of the 4-sphere of radius $r= \sqrt{m}$.  To compute $\beta_\infty(m)$ we need to know the ``volume'' of the 4-ball $B_{4,r}: x^2 + y^2 + z^2 + w^2 \leq r^2$ is given by the well-known formula 
$$\vol(B_{4,r}) = \tfrac{\pi^2}{2} r^4.$$
%
%
(There are many ways to see this, for example as a consequence of Pappus's Centroid Theorem \cite[\S6.18, p166]{Eves:1976bs} once the volume of the $3$-ball $B_{3,r}$ is known to be $\frac{4}{3}\pi r^3$.)
%

We now compute $\beta_\infty(m)$ 
using the 
open sets
$U = U_\ve := (m-\ve, m+\ve)$, giving
\begin{equation}\label{Eq:local_density_at_infty}
\begin{aligned}
\beta_\infty(m) 
&:= \lim_{U \supset \{m\}, U\rightarrow \{m\}}
\frac{\vol_{\R^n} (Q^{-1}(U))}{\vol_{\R}(U)} \\
&= \lim_{\ve\rightarrow 0}
\frac{\vol_{\R^n} (Q^{-1}(U_\ve))}{\vol_{\R}(U_\ve)} \\
&= \lim_{\ve\rightarrow 0}
\frac{\frac{\pi^2}{2}(\sqrt{m+\ve}^4 - \sqrt{m-\ve}^4)}{2\ve} \\
&= \lim_{\ve\rightarrow 0}
\frac{\frac{\pi^2}{2}((m+\ve)^2 - (m-\ve)^2)}{2\ve} \\
&= \lim_{\ve\rightarrow 0}
\frac{\frac{\pi^2}{2}((\cancel{m^2} +2m\ve + \cancel{\ve^2}) - (\cancel{m^2} -2m\ve + \cancel{\ve^2}))}{2\ve} \\
&= \lim_{\ve\rightarrow 0} \frac{\pi^2}{\cancel2} \frac{\cancel{4}m\cancel\ve}{\cancel2\cancel\ve} \\
&=  \pi^2 m. 
\end{aligned}
\end{equation}


\subsection{Understanding $\beta_p(m)$ by counting}

When $v = p$,  we can think of $\Z_p$ as coming from the quotients $\Z/p^i\Z$ where $i$ is very large (i.e. $\Z_p = \plim \Z/p^i\Z$).  Because of this we can interpret the local density $\beta_p(m)$ as a statement about the number of solutions of $Q(\x) \equiv m \pmod{p^i}$ for sufficiently large powers $p^i$.  More precisely, we have 

\begin{lem} \label{Lemma:local_density_by_counting}When $v = p$ is a prime number and $Q(x)$ is a quadratic form in $n$ variables, then we may write $\beta_p(m)$ as 
$$
\beta_p(m) = \lim_{i \rightarrow\infty} 
\frac{\#\{\x \in (\Z/p^i\Z)^n \mid Q(\x)\equiv m \pmod{p^i}\}}{p^{(n-1)i}}.
$$
\end{lem}

\begin{proof}
This follows from the definition by choosing open sets $U_i := p^i \Z_p$.  Then 
$$
\vol_{\Z_p}(U_i) = \frac{1}{p^i}
$$
and each solution $\x$ of $Q(\x)\equiv m \pmod{p^i}$ gives a $p$-adic coset $\x + p^i\Z_p^n$ of solutions in $Q^{-1}(U_i)$.  Therefore since $\vol (p^i\Z_p^n) = \frac{1}{p^{ni}}$, we have
$$
\vol_{\Z_p^n}(Q^{-1}(U_i)) = \frac{1}{p^{ni}}\cdot \#\{\x \in (\Z/p^i\Z)^n \mid Q(\x)\equiv m \pmod{p^i}\}
$$
and so 
$$
\begin{aligned}
\beta_p(m) 
&= \lim_{i \rightarrow\infty}  
\frac{\vol_{\Z_p^n}(Q^{-1}(U_i))}{\vol_{\Z_p}(U_i) }\\
&= \lim_{i \rightarrow\infty} 
\frac{\frac{1}{p^{ni}} \cdot \#\{\x \in (\Z/p^i\Z)^n \mid Q(\x)\equiv m \pmod{p^i}\}}{\frac{1}{p^i}} \\
&= \lim_{i \rightarrow\infty} 
\frac{\#\{\x \in (\Z/p^i\Z)^n \mid Q(\x)\equiv m \pmod{p^i}\}}{p^{(n-1)i}}.
\end{aligned}
$$
\end{proof}
Philosophically we should think of this formula as the number of solutions $\pmod{p^i}$ divided by the ``expected number'' of solutions (based solely on knowing the dimension of $Q(\x)=m$ is $n-1$).
%
%
To see that this limit actually exists, we need to invoke Hensel's lemma which (as a consequence) says that if $i$ is sufficiently large then the sequence defining $\beta_p(m)$ is constant.
%
%
%
%
%
%
%
%
%
%
In particular, for $Q(\x) = x_1^2 + \cdots + x_n^2$ it is enough to compute the (non-zero) solutions $\pmod p$ if $p>2$ and $\pmod 8$ if $p=2$.  In the next few sections we compute the local densities  $\beta_p(m)$ by counting these numbers of solutions.

\subsection{Computing $\beta_p(m)$ for all primes $p$}

Counting solutions to a polynomial equation over finite fields $\Z/p\Z$ can be done explicitly by the method of ``exponential sums'' (sometimes called Gauss sums or Jacobi sums), and this gives particularly simple formulas for degree 2 equations.  One such formula is

\begin{lem} Suppose $p\in \N$ is a prime $>2$, then
$$
r_{x^2+y^2 + z^2 + w^2\!,\, p}(m) = 
\begin{cases}
p^3 - p &\text{ if } p\nmid m,\\
p^3 + p(p-1) &\text{ if } p\mid m.\\
\end{cases}
$$
\end{lem}
\noindent
which gives the following explicit local density formulas:
\begin{lem} \label{Lemma:local_densities_at_p} Suppose $p\in \N$ is a prime $>2$, then
$$
\beta_{x^2+y^2 + z^2 + w^2\!,\, p}(m) = 
\begin{cases}
1 - \frac{1}{p^2} &\text{ if } p\nmid m,\\
\(1 - \frac{1}{p^2}\)\(1 + \frac{1}{p}\) &\text{ if $p\mid m$ but $p^2 \nmid m$}.\\
\end{cases}
$$
\end{lem}

\begin{rem}  The formula when $p\mid m$ follows by counting all solutions except $\x = \vec 0$, since that solution will not lift by Hensel's lemma to a solution of $Q(\x) = m \pmod {p^2}$.
\end{rem}

When $p=2$ we need to understand the local densities $\pmod 8$, which we do by explicitly enumerating the values $Q(\x)$ of all vectors $\x \in (\Z/8\Z)^4$, giving

\begin{lem} \label{Lemma:local_densities_at_2} Suppose $p=2$, then
$$
\beta_{x^2+y^2 + z^2 + w^2\!,\, 2}(m) = 
\begin{cases}
1 &\text{ if } p\nmid m,\\
\frac{3}{2} &\text{ if } p\mid m \text{ but } p^2\nmid m.\\
\end{cases}
$$
\end{lem}

\subsection{Computing $r_Q(m)$ for certain $m$}

We are now in a position to compute the number of representations $r_Q(m)$ for some simple numbers $m$.  To warm up, we see that when $m=1$ we have 
\begin{align}
r_Q(1) 
&= \prod_v \beta_v(1)
= \beta_\infty(1) \, \beta_2(1) \, \prod_{p>2} \beta_p(1) \\
&= \(\pi^2 \cdot 1\) \,(1)\, \prod_{p>2} \(1-\frac{1}{p^2}\) \\
&= \pi^2\, \(\frac{1}{1-\frac{1}{2^2}}\) \, \prod_{p} \(1-\frac{1}{p^2}\) \\
&= \frac{4\pi^2}{3} \, \prod_{p} \(1-\frac{1}{p^2}\) \\
&= \frac{4\pi^2}{3} \, \frac{1}{\zeta(2)} \\
&= \frac{4\cancel{\pi^2}}{3} \, \frac{6}{\cancel{\pi^2}} \\
&= 8
\end{align}
which we could also have worked out (perhaps more quickly) by observing that if $Q(\x) = \sum_{i=1}^4 x_i^2 = 1$, then we must have $|x_i| \leq 1$ and at all but one $x_i$ is zero.  

Now suppose that $m = p > 2$ is prime.  Then our computation at of $r_Q(p)$ looks almost the same as when $m=1$ with the exception that the factors at $v=\infty$ and $v=p$ have changed.  This gives
\begin{align}
r_Q(p) 
&= r_Q(1) \cdot \frac{\beta_\infty(p)}{\beta_\infty(1)} \cdot \frac{\beta_p(p)}{\beta_p(1)} \\
&= r_Q(1) 
\cdot \frac{\cancel{\pi^2} p}{\cancel{\pi^2}} 
\cdot \frac{\cancel{\(1 - \frac{1}{p^2}\)}\(1 + \frac{1}{p}\)}{\cancel{\(1 - \frac{1}{p^2}\)}} \\
&= r_Q(1) \cdot p \cdot \(1 + \tfrac{1}{p}\) \\
&= 8 (p+1)
\end{align}

Finally, we suppose that $m$ is an odd squarefree number $t$.  Then the computation changes at $v=\infty$ 
and at all primes $p\mid t$,
giving
\begin{align}
r_Q(t) 
&= r_Q(1) \cdot \frac{\beta_\infty(t)}{\beta_\infty(1)} \cdot \prod_{p \mid t} \frac{\beta_p(t)}{\beta_p(1)} \\
&= r_Q(1) \cdot t \cdot \prod_{p \mid t} \tfrac{p+1}{p} \\
&= 8 \prod_{p \mid t} (p+1).
\end{align}
We see that this agrees with Jacobi's divisor sum formula for $r_Q(m)$ in Theorem \ref{Thm:Sum_of_four_squares} since the positive divisors of $t$ are exactly the terms appearing when the product $\prod_{p \mid t} (p+1)$ is fully expanded.
One could continue to prove Jacobi's formula for $r_Q(m)$ for any $m\in\N$ by extending this computation, though the computations of the local densities $\beta_p(m)$ when $p=2$ and at primes where $p^2\mid m$ become somewhat more involved.



%



\chapter{Quaternions and Clifford Algebras} \label{Chapter:Clifford}

In this chapter, we describe some important algebraic structures naturally associated with quadratic forms.  One of them is the Clifford algebra, which one can think of an algebra that enhances a quadratic space with a multiplication law.  The other is the Spin group, which is an  algebraic group that is the ``double cover'' of the special orthogonal group and can be constructed naturally in terms of the Clifford algebra.




\section{Definitions}
Quadratic forms are closely connected with quadratic extensions, both those which are commutative (quadratic fields and their rings of integers) and also non-commutative (quaternion algebras and their maximal orders).  
We now
explore some connections with non-commutative algebras of a particularly nice kind (known as ``central simple algebras''), and describe their basic structure.  Good references for central simple algebras are \cite[\S4.6]{Jacobson:1989fu}, 
\cite[\S1--2]{Gille:2006tz}, \cite[Ch III--IV]{Lam:2005kl} and \cite[Ch IV]{Shimura:2010vn}.

\bigskip

We begin by defining a {\bf central simple algebra} $A$ as a finite-dimensional (possibly non-commutative) algebra over a field $k$ whose center is $k$ and which contains no proper non-zero two-sided ideals.   To make the dependence on $k$ explicit, we sometimes write $A$ as $A/k$.  We say that the {\bf dimension} of $A/k$ is the dimension of $A$ as a vector space over $k$.  

\begin{thm}
Suppose that $A_{1}$ and $A_{2}$ are central simple algebras over $k$.  Then the tensor product $A_{1} \otimes_{k} A_{2}$ is also a central simple algebra over $k$.
\end{thm}
\begin{proof}
See \cite[Cor 3, p219]{Jacobson:1989fu}.
\end{proof}

Another nice property of central simple algebras is that we can freely extend the base field $k$ and preserve the property of being central simple (though now with a larger center!):
\begin{thm}
Suppose that $A/k$ is a central simple algebra
 and $K$ is a field containing $k$, then $A/K := A \otimes_{k} K$ is a central simple algebra over $K$ of the same dimension as $A/k$.
\end{thm}
\begin{proof}
See \cite[Cor 2, p219]{Jacobson:1989fu} and the discussion on the top of p220.
\end{proof}

The simplest examples of central simple algebras are the matrix algebras $M_{n}(k)$ (which have dimension $n^{2}$).  Notice that any central simple algebra over $k$ which is commutative must be just $k$ itself, so in general central simple algebras are non-commutative.  
The next simplest example of a central simple algebra which is not a field (i.e. non-commutative) is called a {\bf quaternion algebra}, and can be defined in terms of a basis $\mathcal{B}  = \{1, i, j, \kappa \}$ satisfying the relations $i^{2} = a$, $j^{2} = b$, $\kappa := ij = -ji$ for some fixed $a,b\in k^{\times}$ (where we always assume that $\Char(k) \neq 2$).  This quaternion algebra is often referred to by the symbol $\quatalg{a}{b}{k}$, though various different choices of $a$ and $b$ may give rise to isomorphic quaternion algebras (e.g. $\quatalg{1}{-1}{k} \cong \quatalg{4}{-1}{k}$).

If $A/k \cong M_{n}(k)$ for some $n$, we say that $A$ is {\bf split}.  If it happens that $A \otimes_{k} K$ is split for some extension $K$ of $k$, we say that $A/{k}$ is {\bf split by $K$}, or that $K$ is a {\bf splitting field} for $A/k$.  The following theorem (and proof) shows that it is not too difficult to find a splitting field for any $A/k$:
%
%
\begin{thm} \label{Thm:Splitting_fields}
If $A/{k}$ is a central simple algebra over $k$, then $A/{k}$ is split by some finite separable extension $K/k$.  
\end{thm}
\begin{proof}
The existence of a finite extension splitting $A$ follows from \cite[Thrm 4.8, p221]{Jacobson:1989fu} and the discussion on the top of p220.  To see that they are not hard to construct explicitly, see \cite[Thrm 4.12, p224]{Jacobson:1989fu}.  Finally separability of the extension follows from (the proof of) \cite[Prop 2.2.5, p22]{Gille:2006tz}.
\end{proof}

Since base change doesn't change the dimension of a central simple algebra, and we can always enlarge our base field so that $A$ splits, we have the following useful corollary and definition:
\begin{cor} 
The dimension of a central simple algebra is always a square.
\end{cor}

\begin{defn}
If $A/k$ has dimension $n^2$, then we say that $A$ has {\bf degree} $n$.
\end{defn}

We can use this idea to define a norm map 
$N_{A/k}:A \ra k$
by extending scalars to the separable closure $k^\text{sep}$, which splits $A$ by Theorem \ref{Thm:Splitting_fields}, giving an isomorphism $A/k^\text{sep} \cong M_{n}(k^\text{sep})$.
%
%
We then define the {\bf norm} $N_{A/k}(x)$ as the determinant of $x$ under this isomorphism.  Since $\det(x)$ is constant on conjugacy classes, the norm is independent of the choice of isomorphism, and is invariant under the Galois action as well, hence is in $k$.  Since the determinant is multiplicative, we see that 
$$
N_{A/k}(\al\beta) = N_{A/k}(\al) N_{A/k}(\beta) \qquad \text{for all $\al, \beta \in A$.}
$$

If it happens that every non-zero element of $A$ is invertible (i.e. $\al\in A- \{0\} \implies$ there is some $\al' \in A$ so that $\al \al' = 1$ and $\al' \al=1$) then we say that $A$ is a {\bf division algebra}. One can think of division algebras as natural non-commutative generalizations of (finite degree) field extensions $K/k$. In fact any non-zero element $\al$ of a central simple algebra $A$ of degree $n$ generates a commutative subalgebra $k[\al]\subseteq A$ of degree $[k[\al]:k]$ dividing $n$.
In the case of a quaternion algebra one can realize the norm map in terms of a conjugation operation explicitly on the basis (by taking $\al = a + bi + cj + d \kappa \mapsto \bar\al := a - bi -cj -d \kappa$), giving the norm as $N_{A/k}(\al) = \al \bar\al$.  
The property of being a division algebra can be easily characterized in terms of the norm map.
\begin{thm} \label{Thm:division_alg_by_norm}
A central simple algebra $A$ over $k$ is a division algebra iff the condition $N_{A/k}(\al) = 0 \iff \al = 0$ holds.
\end{thm}

\begin{proof}
Notice 
that $\al$ is invertible in $A$ $\iff$ the left multiplication map $L_\al:A/k \ra A/k$ is an invertible linear map (by taking $\al^{-1}:= L_\al^{-1}(1)$).  However $L_\al$ is invertible $\iff$ 
its linear extension $L_\al^\text{sep}:A/k^\text{sep} \ra A/k^\text{sep}$ over the separable closure $k^\text{sep}$ of $k$ is invertible, 
which happens 
iff $\det(L_\al^\text{sep}) = N_{A/k}(\al)^n  \neq 0$,
where $n$ is the degree of $A$ over $k$.  
(See also \cite[\textsection16.3, Cor. a, p300]{Pierce:1982fk}.)
%
%
%
%

In the special case where $A$ is quaternion algebra this follows more directly by noticing that if  
$\al$ is invertible then its unique two-sided inverse has the form $\al^{-1} = {\bar{\al}} \cdot ({N_{A/k}(\al)})^{-1}$, which exists iff $N_{A/k}(\al)\neq 0$.
\end{proof}

The following important structural result of Wedderburn shows that division algebras play a crucial role in the study of central simple algebras.  It is also the starting point for defining the Brauer group, which 
we will not discuss here, but 
is discussed in Parimala's lecture notes \cite{Parimala} in this volume.
\begin{thm}[Wedderburn] \label{Thm:Wedderburn}
Every central simple algebra $A$ over $k$ is isomorphic to a matrix ring over a division algebra, i.e.
$$
A \cong  M_{n}(D)
$$
where $D$ is a (unique) division algebra over $k$, and $n \in \N$ is the degree of $A/k$.
\end{thm}

\begin{proof}
See \cite[Thrm 2.1.3, p18]{Gille:2006tz}.
\end{proof}

We now specialize to consider quaternion algebras, which are very closely related to quadratic spaces and questions about quadratic forms.  One important connection is given by considering the {\bf associated quadratic space} $(V,Q) := (A/k, N_{A/k})$ of the quaternion algebra $A/k$.

\begin{lem} \label{Lem:quat_assoc_quad_space}
A quaternion algebra $A/k$ is uniquely determined (up to isomorphism) by its associated quadratic space.
\end{lem}

\begin{proof}
When $\Char(k) \neq 2$ this is \cite[Thrm 2.5(a)-(b), pp57--8]{Lam:2005kl}, and more generally this follows from \cite[Ch V, Prop 2.4.1, p256]{Knus:1991qa}.
\end{proof}

%
\noindent
In this language we have the following useful corollary of Theorem \ref{Thm:division_alg_by_norm}.
\begin{cor} 
A quaternion algebra $A/k$ is a division algebra iff  its associated (4-dimensional) quadratic space 
 is anisotropic.
\end{cor}


%
%

%
%
\noindent
By combining this with the
theory of local invariants of quadratic spaces in \S\ref{Sec:Local_quadratic_spaces}, we have the following uniqueness result:

\begin{thm}
There is a unique quaternion division algebra over each of the local fields $\Q_{p}$ and $\R$.
\end{thm}

\begin{proof}
Over $\R$ we see that $\quatalg{a}{b}{\R}$ is determined by the signs of $a$ and $b$, and that this is split iff at least one of them is $>0$.  The remaining case gives $a=b=-1$, which gives the Hamiltonian quaternions $\H$ and is the unique division algebra over $\R$.

Over $\Q_p$ this follows from Lemma \ref{Lem:quat_assoc_quad_space} and the fact that there is a unique 4-dimensional anisotropic quadratic space over $\Q_p$ (characterized by the Hilbert symbol relation $c_p = (-1, -d_p)_p$) \cite[Lemma 2.6, p59]{Cassels:1978aa}.
\end{proof}

Therefore, since every non-split quaternion algebra is a division algebra we see that

\begin{thm} \label{Thm:two_local_quat_algs}
There are exactly two quaternion algebras (up to isomorphism) over each of the local fields $k = \Q_{p}$ or $\R$:
 the split algebra $M_{2}(k)$, and a division algebra $D$.
\end{thm}

When $k = \Q_{p}$ or $\R$, the dichotomy of Theorem \ref{Thm:two_local_quat_algs} is often referred to as saying that a quaternion algebra $A/k$ is either {\bf split} or {\bf ramified} (when it is a division algebra).  The term ``ramified'' is used here because in the associated valuation theory of local division algebras (which is discussed in \cite[\S21, particularly Thrm 21.17, p108]{Shimura:2010vn}), the division quaternion algebra $D$ has a valuation ring with maximal ideal $\p$ satisfying $\p^2 = (p) := p\Z$, which agrees with the usual notion of ramification in algebraic number theory.

To decide whether the local quaternion algebra $A/k$ above is split or ramified, one can use the easily computable {\bf (local) Hilbert symbol}
\begin{equation} \label{Eq:Hilbert_symbol_defn}
(\cdot, \cdot)_v : \Q_v^\times / (\Q_v^\times)^2 \times  \Q_v^\times / (\Q_v^\times)^2 \longrightarrow \{\pm1\}
\end{equation}
which is a non-degenerate multiplicative symmetric bilinear form on the (non-zero) squareclasses of $\Q_v$.  The Hilbert symbol arises naturally in the study of Class Field Theory \cite[Ch V, \S3, with $n=2$]{Neukirch:1999mi}, and is defined by the (not obviously symmetric) relation $(a,b)_v = 1 \iff a \in N_{K_v/\Q_v}(K_v^\times)$ where $K_v := \Q_v(\sqrt{b})$.  The Hilbert symbol has many interesting properties:

\begin{thm}
The local Hilbert symbol defined in (\ref{Eq:Hilbert_symbol_defn}) satisfies the following properties:
\begin{enumerate}
\item $(a,b)_v = (b, a)_v$,  (symmetry)
\item $(a_1 a_2, b)_v = (a_1, b)_v (a_2, b)_v$, (bilinearity)
\item $(a, b)_v = 1$ for all $b \in \Q_v^\times / (\Q_v^\times)^2 \implies a \in (\Q_v^\times)^2$, (non-degeneracy)
\item $(a, -a)_v = (a, 1-a)_v =1$,  (symbol)
\item $(a,b)_p = 1$ if $p\neq 2$ and $\ord_p(a), \ord_p(b) \in 2\Z$.
\end{enumerate}
\end{thm}

\begin{proof}
This follows from \cite[Lem 2.1, 42]{Cassels:1978aa} except for $(a, 1-a)_v =1$, which follows since $1-a  =  N_{\Q_v(\sqrt{a})/\Q_v}(1 + \sqrt{a}).$  See also \cite[Ch V, Prop 3.2, p334]{Neukirch:1999mi} for the anaogous proofs over number fields.
\end{proof}

\noindent
Hilbert symbols are
also 
an example of a ``symbol'' in the sense of $K$-theory (see \cite[Ch V, \S6 and Ch X, \S6, p362]{Lam:2005kl} and \cite[Ch VI, \S4, p356]{Neukirch:2008pi}), but for our purposes it is enough to be able to explicitly compute them, which can be done with the tables on \cite[pp43--44]{Cassels:1978aa}.  The question of computing Hilbert symbols (and splitting of quaternion algebras) over number fields is discussed in Voight's paper \cite{VOIGHT} in this volume.

\medskip
Finally, the Hilbert symbol also satisfies the global ``reciprocity'' relation, from which quadratic reciprocity can be easily proved.
\begin{thm}  For all $a,b \in \Q^{\times}$, we have the product formula
$$
\prod_{v} (a,b)_{v} = 1, 
$$
and all but finitely many factors are one.
\end{thm}

\begin{proof}
See \cite[Lem 3.4, 46]{Cassels:1978aa} or \cite[Ch VI, Thrm 8.1, p414]{Neukirch:1999mi} for the analogous result over number fields.
\end{proof}

\noindent
This theorem has the following important parity consequence for quaternion algebras $A/\Q$.
\begin{cor}
Given any quaternion algebra $A/\Q$, the set of places $v$ where $A/\Q_v$ is ramified has even cardinality.
\end{cor}

\begin{proof}
By writing $A/\Q$ as $\quatalg{a}{b}{\Q}$ for some $a,b\in \Q^\times$, we see that $A/\Q_v$ is ramified $\iff (a,b)_v = -1$, and the product formula guarantees this happens an even number of times.
\end{proof}

\section{The Clifford Algebra}

Good basic references for Clifford algebras over fields of characteristic $\neq 2$ are \cite[\S4.8]{Jacobson:1989fu}, \cite[Ch V]{Lam:2005kl}, \cite[Ch 10]{Cassels:1978aa}.  The valuation theory of central simple algebras over number fields can be found in \cite{Shimura:2010vn}, and Clifford algebras over general rings are discussed thoroughly in \cite[Ch IV--V]{Knus:1991qa}.  A very in-depth treatment of Clifford algebras as well as automorphic forms on their associated Spin groups can be found in the recent book of Shimura \cite{Shimura:2004qe}.

\bigskip

Given a quadratic space $(V,Q)$ over a field $K$ (of characteristic $\neq 2$) of dimension $n$, we define its {\bf Clifford algebra} $C(V)$ as the $K$-algebra generated by all formal multiplications of scalars $k \in K$ and vectors $\v \in V$ subject to the family of ``squaring relations'' that $\v\,^2 = \v \cdot \v = Q(\v)$ for all $\v \in V$.  More formally, we can construct the Clifford algebra as a quotient $C(V) := T(V)/I(V)$ of the tensor algebra $T(V) = \oplus_{i=0}^\infty (\otimes^i V)$ by the ideal of relations 
$$
I(V) := 
\begin{matrix}
\text{the ideal of $T(V)$ generated by the set } \\ 
\{\v\,{}^2 - Q(\v),\,\,  k\cdot \v - k\v
 \,\,\text{ for all } \v \in V, k \in K\}.
\end{matrix}
$$
This shows that $C(V)$ is well-defined (and we will soon see that it is non-zero!).

One useful observation is that 
there is also a
nice relationship between multiplication in $C(V)$ and the inner product $B(\v,\w)$.  We see this by expanding out 
\begin{align}
Q(\v+\w) 
&= (\v + \w)^2 \\
&= \v\,^2 + \v \cdot \w + \w \cdot \v + \w\,^2 \\
&= Q(\v) + (\v \cdot \w + \w \cdot \v) + Q(\w)
\end{align}
and comparing this with the polarization identity (\ref{Eq:Polarization_id}), which shows that  
\begin{equation}
\v \cdot \w + \w \cdot \v = 2B(\v, \w).
\end{equation}
This relation can be used to give a unique presentation of any element $\al \in C(V)$ in terms of a given choice of basis $\B = \{\v_1, \dots, \v_n\}$ of $V$, since we can reverse the order of adjacent elements to present them in terms of the basis of all possible products of distinct vectors $\v_i \in \B$ with indices $i$ in increasing order.
Because these products are indexed by the $2^n$ subsets $I$ of $\{1, \dots, n\}$, we see that 
\begin{thm}
The dimension of the Clifford algebra is  $\dim_K(C_0(V)) = 2^{n}$.
\end{thm}

An interesting special case of the Clifford algebra is when the quadratic form $Q$ is identically zero.  In this case, by taking a basis $\B$ for $V$ and applying the relations above we see that $\v_i \cdot \v_j = -\v_j \cdot \v_i$ when $i \neq j$ and $\v_i\,^2 = 0$. This shows that $C(V)$ is just the exterior algebra $\oplus_i (\bigwedge^i V)$.  When $Q$ is not identically zero we can think of $C(V)$ as a deformation of the exterior algebra that encodes the arithmetic of $Q$.

\medskip
Another interesting fact about the Clifford algebra is that it has a natural $(\Z/2\Z)$-grading (called the {\bf parity}) coming from the $\Z$-grading on the tensor algebra $T(V)$ and the fact that the relations in $I(V)$ only involve relations among elements of the same parity.  We say that an element of $C(V)$ is said to be {\bf even} or {\bf odd} if it can be written as a sum of elements of $T(V)$ of even or odd degree respectively.  Given this, we can consider the subalgebra $C_0(V)$ of even elements in $C(V)$, called the {\bf even Clifford algebra} of $V$.  It follows from our basis description of $C(V)$ above and simple facts about binomial coefficients that 
%
\begin{thm}
The dimension of the even Clifford algebra is  $\dim_K(C_0(V)) = 2^{n-1}$.
\end{thm}

Both $C(V)$ and $C_0(V)$ have a {\bf canonical involution} $\al \mapsto \widetilde{\al}$ defined (on the pure tensor elements) by reversing the order of every product of vectors, i.e.
$$
\al := \v_1 \cdots \v_k \longmapsto  \v_k \cdots \v_1 =: \widetilde{\al},
$$
and then extending this map linearly to the entire algebra.  
We can use this involution to define a multiplicative {\bf norm function} 
$N:C(V)\ra K$ by the product
$$
N(\al) := \al \cdot \widetilde{\al}.
$$
To see that $N(\al)\in K$ notice that if $\al := \v_1 \cdots \v_k$ then 
\begin{equation} \label{Eq:norm_of_vector_product}
N(\al) 
= \al \cdot \widetilde{\al} 
= (\v_1 \cdots \v_k) \cdot (\v_k \cdots \v_1) 
= Q(\v_k) \cdots Q(\v_1) \in K,
\end{equation}
by repeatedly using the relation $\v_i\,^2 = Q(\v_i)$.  This also shows that $N(\al) = \widetilde{\al} \cdot \al$.

\medskip
We will be interested in multiplicative subgroups of $C(V)$, so we say that an element $\al \in C(V)$ is {\bf invertible} if there is some $\al^{-1} \in C(V)$ so that $\al \cdot \al^{-1} = 1$.  
Notice that if $\al$ is invertible, then 
%
$$
\al^{-1} = \frac{\widetilde{\al}}{N(\al)}
$$
and we also have $\al^{-1} \cdot \al = 1$, so our inverses are ``two-sided''.
%
It is also useful to notice that
we already have a good understanding of what vectors $\v \in V$ are invertible.

\begin{lem}
Suppose that $\v \in V \subset C(V)$.  Then $\v$ is invertible $\iff \v$ is anisotropic. 
\end{lem}

\begin{proof}
This follows since $\v$ is invertible $\iff N(\v) = \v\cdot \v = Q(\v) \neq 0$.
\end{proof}


\medskip
Finally, we mention the following theorem that explains the structure of the Clifford algebra as a central simple algebra.

\begin{thm}
Suppose that $V$ is a non-degenerate quadratic space of dimension $n$.  Then $C(V)$ is a central simple algebra if $n$ is even and $C_0(V)$ is a central simple algebra when $n$ is odd.
\end{thm}

\begin{proof}
See  \cite[Thrm 4.14, p237]{Jacobson:1989fu}, \cite[Thrm 2.4 and 2.5, p110]{Lam:2005kl},  \cite[Thrm 23.8, p125]{Shimura:2010vn}.  A more general version of this holds over rings, where the word ``Azumaya'' replaces ``central simple''.  (See \cite[Ch IV, Thrm 2.2.3, p203 and  Thrm 3.2.4(1), p210]{Knus:1991qa} for proofs, and \cite[Ch 2]{Saltman:1999or} for a discussion of Azumaya.)
\end{proof}

\section{Connecting algebra and geometry in the orthogonal group}

One important feature of the orthogonal group $O(V)$ is that can be used to describe equivalences of quadratic spaces and quadratic lattices, however this is not very useful unless one can somehow describe the elements of $O(V)$.  One approach for doing this is to try to use the geometry of $V$ to construct explicit transformations in $O(V)$.  We now describe how this is done (over fields $K$ of characteristic $\Char(K) \neq 2$).

Given some $\v \in V$ 
with
$Q(\v) \neq 0$, we can define the {\bf reflection symmetry} $\tau_{\v} \in O(V)$ defined by sending $\v \mapsto -\v$ and pointwise fixing all vectors in the orthogonal complement $(K\v)^\perp$ of the line spanned by $\v$.  Using the standard projection formulas of linear algebra (e.g. \cite[\textsection4.2]{Strang_linear_algebra_4}) we see that $\tau_{\v}$ can be written explicitly as 
\begin{equation} \label{Eq:reflection_defn}
\tau_{\v}(\w) = \w - 2\frac{B(\v, \w)}{B(\v,\v)} \v,
\end{equation}
which is only defined if $B(\v,\v) = Q(\v) \neq 0$.  Notice that since we are reversing the direction of a line, and stabilizing its complement, we know $\det(\tau_{\v}) = -1$.
One useful property of these reflection symmetries is that they can be explicitly seen to act transitively on vectors of a given non-zero length.  More precisely, 
\begin{lem}
Suppose that $\v, \w \in V$ satisfy $Q(\v) = Q(\w)$.
\begin{enumerate}
\item[a)]
If $Q(\v) = Q(\w)\neq 0$.  Then  $\al\v = \w$ where $\al$ is a product of at most two reflection symmetries.
\item[b)]
If $Q(\v-\w)\neq 0$, then
$$
\tau_{\v- \w}(\v) = \w.
$$
%
%
\end{enumerate}
\end{lem}

\begin{proof}
This can be found in \cite[pp19--20]{Cassels:1978aa} among other places, though we give the argument here.
Part b) follows from a direct computation with (\ref{Eq:reflection_defn}).  Part a) follows from b) if $Q(\v-\w) \neq 0$, otherwise the polarization identity (\ref{Eq:Polarization_id}) ensures that $Q(\v + \w) \neq 0$, and so part b) allows us to find a symmetry interchanging $\v$ and $-\w$.      From here the symmetry $\tau_{\w}$ interchanges $\w$ and $-\w$, giving $\al = \tau_{\w} \cdot \tau_{\v + \w}$.
\end{proof}

This transitive action of reflection symmetries shows that they generate the full orthogonal group $O(V)$ for any non-degenerate quadratic space.


\begin{thm} \label{Thm:generation_by_symmetries}
If $(V,Q)$ is a non-degenerate quadratic space, then every element $\beta \in O(V)$ can be written as a product of reflection symmetries, i.e. 
$$
\beta = \tau_{\v_1} \cdots \tau_{\v_k}
$$
for some vectors $\v_i \in V$ with $Q(\v_i) \neq 0$.  
%
\end{thm}

\begin{proof}
This is proved \cite[Lemma 4.3, pp20--21]{Cassels:1978aa}.  This follows by induction on the dimension on $V$, since for any $\v$ with $Q(\v)\neq 0$ we can find some product of symmetries $\al$ so that $\al \v = \beta \v$.  Therefore $\al^{-1} \beta$ fixes $K\v$ and also $W := (K\v)^\perp$, and we are reduced to showing that $\al^{-1} \beta$ is a product of symmetries on $W$.  When $\dim(W) = 1$, this holds because $\beta:\v \mapsto \pm \v$, completing the proof.
\end{proof}

There is a particularly interesting map called the {\bf spinor norm map}, denoted $\text{sn}(\al),$ from $O^+(V)$ to the squareclasses $K^\times/(K^\times)^2$ that can be defined easily by using 
the reflection symmetry description of $O(V)$.  To do this we write $\al \in O^+(V)$ as a product of symmetries $\tau_{\v}$ and define
\begin{align}
\al = \, &\tau_{\v_1} \cdots \tau_{\v_k} \longmapsto Q(\v_1) \cdots Q(\v_k) =: \text{sn}(\al).
\end{align}
This gives a squareclass because we could have rescaled any of the $\v_i$ without changing its associated symmetry $\tau_{\v_i}$, but we would change $Q(\v_i)$ by a non-zero square.  However it is more work to show that $\text{sn}(\al)$ is independent of our particular presentation of $\al$ as a product of symmetries.

\begin{lem} \label{Lem:scalar_center}
Suppose that $(V,Q)$ is a non-degenerate quadratic space.  Then any even element in the center of $C(V)$ is a scalar.
\end{lem}

\begin{proof}
This can be shown by taking an orthogonal basis $\{\vec e_i\}$ for $V$, which necessarily satisfies $\vec e_i \vec e_j =  -\vec e_j \vec e_i$, and imposing the commutation relation.  This is done explicitly in \cite[Lemma 2.3, p174]{Cassels:1978aa}, \cite[\S54:4, pp135--6]{OMeara:1971zr}, \cite[Cor 23.9, p126]{Shimura:2010vn} and somewhat indirectly in \cite[Thrm 3.4, p92 and Thrm 2.2, p109]{Lam:2005kl}. 
\end{proof}

\begin{thm}
Suppose that $(V,Q)$ is a non-degenerate quadratic space.  Then the spinor norm map $\text{sn}:O(V) \ra \K^\times/(\K^\times)^2$ is a well-defined group homomorphism.
\end{thm}

\begin{proof}
To see that $\text{sn}(\al)$ is well-defined, notice that any two expressions for $\al$ as a product of transpositions gives rise to an expression for the identity map as a product of an even number of reflection symmetries
$$
\prod_i \tau_{\v_i} = \id \in SO(V),
$$
and $\text{sn}(\al)$ is well-defined iff $\prod_i Q(\v_i) \in (\K^\times)^2$.  Lemma \ref{Lem:clifford_conjugation_symmetry} allows us to interpret  $\tau_{\v}$ as conjugation by $\v$ in $C(V)$ and  letting $u := \prod_i \v_i \in C_0(V)$ we see that $u\w u^{-1} = \w$ for all $\w \in V$.  Therefore by Lemma \ref{Lem:scalar_center} we know that $u \in K^\times$, and so 
$\prod_i Q(\v_i) = u \widetilde{u} = u^2 \in (K^\times)^2$.
%
%
%
%
%
%
%
(This argument also appears in \cite[Cor 3, p178]{Cassels:1978aa}, \cite[\S24.8, p131]{Shimura:2010vn}, \cite[\S55, p137]{OMeara:1971zr} and \cite[Thrm 1.13, p108]{Lam:2005kl}.)
\end{proof}

%
%
%
%
%
%

\section{The Spin Group}

Now that we understand some basic properties of the Clifford algebra and the orthogonal group, we are ready to construct a very useful ``two-fold cover'' 
of the special orthogonal group $SO(V)$ called the spin group.  Aside from being interesting in its own right, the spin group plays a very important role in the theory of indefinite quadratic forms.

\medskip





As a first step, we notice that conjugation in the Clifford algebra is a very interesting operation because it naturally produces elements of the orthogonal group.  
For example, if it happens that $u\in C(V)^\times$ satisfies $u^{-1} V u \subseteq V$ then we claim that this conjugation gives an isometry of $V$, and so it is an element of the orthogonal group $O(V)$.  To see this, for any $\x \in V$  we compute 
\begin{align*}
Q(u^{-1} \x u) 
= (u^{-1} \x u) (u^{-1} \x u)
= u^{-1} \cdot \x \cdot \x \cdot u
= Q(\x).
\end{align*}
Amazingly, we can even identify exactly 
which element of $O(V)$ this conjugation gives us.
\begin{lem} \label{Lem:clifford_conjugation_symmetry}
Suppose that $u\in V$ satisfies $Q(u) \neq 0$ and that for all $\x\in V$ the conjugation $\varphi_u: \x \mapsto u^{-1} \x u \in V$.  Then $\varphi_u \in O(V)$ and $\varphi_u$ gives the negative reflection symmetry $-\tau_u$.
\end{lem}

\begin{proof}  We have already seen that $\varphi_u \in O(V)$, so we only need to identify $\varphi_u$ explicitly as 
\begin{align*}
\varphi_u 
& = u^{-1} \x u 
= \frac{u}{Q(u)} \x u \\
&= \frac{1}{Q(u)} \[(u\x + \x u)u - xu^2\] \\
&= \frac{1}{Q(u)} \[2B(x,u) u - xQ(u)\] \\
&= -x + \frac{2B(x,u)}{Q(u)}u \\
&= - \tau_u(\x).
\end{align*}
\end{proof}

This leads us to define the multiplicative subgroup 
$$
U_0 := \{u \in (C_0(V))^\times \mid u^{-1} V u \subseteq V\} \subseteq C_0(V)
$$
on which we have a natural {\bf conjugation map} $\varphi: U_0 \ra O(V)$ defined by sending $u \mapsto (\varphi_u:\x \mapsto u^{-1}\x u)$.  It takes a little work to see that the image of this map is in $SO(V)$.

\begin{lem} \label{Lem:conjugation_into_SO}
The conjugation map above gives a homomorphism $\varphi: U_0 \ra SO(V)$.
\end{lem}

\begin{proof}
If $u\in U_0$ then
by Theorem \ref{Thm:generation_by_symmetries} we can find $r$ anisotropic vectors $\v_i$ so that $\al := \prod_i\v_i$ gives $\varphi(\al) = \varphi(u)$, and so $\beta := \al \cdot u^{-1}$ has $\varphi(\beta) = \id \in O(V)$.  This is equivalent to the commutation relation $\beta\v = \v \beta$ for all $\v \in V$.  However by expressing $\beta$ as a unique linear combination of ordered monomials with respect to some fixed orthogonal basis $\{\w_i\}$for $V$, this commutation relation for $\w_i$ says that each monomial containing $\w_i$ must have even degree, and so $\beta \in C_0(V)$.  Therefore $\al \in C_0(V)$, $r$ is even and $\varphi(u) = \varphi(\al) \in SO(V)$.

This argument can be found in \cite[Thrm 3.1, pp176--7]{Cassels:1978aa}, \cite[Thrm 24.6, pp129--130]{Shimura:2010vn}, and there Shimura points out that this result is originally due to Lipschitz \cite{:1959fk}, though a special case was shown by Clifford.
\end{proof}

%
This map is very useful for connecting the Clifford algebra and the special orthogonal group, as it provides a natural and explicit covering.

\begin{lem} \label{Lem:spin_covering_lemma}
Suppose that $(V, Q)$ is non-degenerate quadratic space.  Then the conjugation map $\varphi: U_0 \ra SO(V)$ is surjective with kernel $K^\times$, and so $U_0/(K^\times) \xrightarrow{\sim} SO(V)$.
\end{lem}
\begin{proof}
Surjectivity follows from Theorem \ref{Thm:generation_by_symmetries} and Lemma \ref{Lem:clifford_conjugation_symmetry}.  To see that $\Ker(\varphi) = K^\times$, use Theorem \ref{Thm:generation_by_symmetries} and Lemma \ref{Lem:scalar_center}.

This can also be found in \cite[Thrm 3.1, p176]{Cassels:1978aa} and \cite[Thrm 24.6(iii), p129]{Shimura:2010vn}.
%
\end{proof}

Using this Lemma, we define the {\bf spin group} $\Spin(V)$ as the elements of $\al \in U_0$ with norm $N(\al) = 1$.  The spin group an algebraic section of the covering map $\varphi: U_0 \ra SO(V)$, which we soon show is a ``double covering'' of its image.
%
%
A helpful observation for doing this is
that the spinor norm of an element of $\Spin(V)$ under this composition can be computed fairly easily.
\begin{lem}
Suppose that $(V, Q)$ is non-degenerate quadratic space.  Then for any $u \in U_0$ we have $sn(u) = N(u) (K^\times)^2$.
\end{lem}

\begin{proof}
This is shown in \cite[Cor 1, p177]{Cassels:1978aa} and \cite[\S24.8 above (24.7a), p131]{Shimura:2010vn}, but we give the argument below.

We first show that any $\al \in U_0$ can be written as a product of an even number of anisotropic vectors $\v_i \in  V$.  To see this, we use the proof of Lemma \ref{Lem:conjugation_into_SO} to see that $\al := \prod_i\v_i \in (C_0(V))^\times$ and that $\varphi(\al \cdot u^{-1}) = \id \in O(V)$.  Therefore $\al \cdot u^{-1}$ commutes with $V$, hence is in the center of $C(V)$, and applying Lemma \ref{Lem:scalar_center} shows that $u  = c \prod_i \v_i$ for some $c \in K^\times$.

Given that $u \in U_0$ can be written as a product of anisotropic vectors $u = \v_1 \cdots \v_r$, the lemma follows from computing
$N(u) = \prod_{i=1}^r Q(\v_i) = sn(u)$ 
using (\ref{Eq:norm_of_vector_product}).
\end{proof}

From this it follows that the spinor norm of the image of any element of $\Spin(V)$ under $\varphi$ must be trivial (i.e. $sn(\varphi(\Spin(V)) = (K^\times)^2$ ), and so the image of $\Spin(V)$ is contained in the {\bf spinor kernel} $\kappa(V) := \ker(sn)$.  In fact  
%

\begin{lem}
The map $\varphi: \Spin(V) \ra SO(V)$ has image $\kappa(V)$ and kernel $\{\pm 1\}$.
The image $\Im(\varphi(\Spin(V))) = \kappa(V)$ since any element $\al \in U_0$ with $sn(\al) = (K^\times)^2$ must have $N(\al) = 1$, and also $\Ker(\varphi(\Spin(V))) = \{\pm 1\}$.  
\end{lem}

\begin{proof}
The image is $\kappa(V)$ because any $\al \in U_0$ with $sn(\al) = (K^\times)^2$ must have $N(\al) = 1$, and the kernel consists of the elements $c \in \K^\times$ with $N(c) = c^2 = 1$.
\end{proof}

We can conveniently summarize our results in the following exact commutative diagram:



$$
\xymatrix{
& \{\pm1\} \ar@{^{(}->}[d] & K^\times \ar@{^{(}->}[d] \\
1 \ar[r] & \Spin(V) \ar[r] \ar@{-->>}[d]^{\varphi} & U_0 \ar[r]^{\text{Norm}} \ar@{->>}[d]^{\varphi} & K^\times \ar[r] \ar[d]^{\text{id}} & 1\\
1 \ar[r] & \kappa(V) \ar[r] & SO(V) \ar[r]^{sn} & K^\times/(K^\times)^2 
}
$$

This shows our main result
\begin{thm}
$\Spin(V)$ is a double covering of the spinor kernel $\kappa(V) \subseteq SO(V)$, and the obstruction to this being a covering map are presence of non-trivial squareclasses of $K^\times$ in the image of the spinor norm map $sn$.
\end{thm}

In the special case where $K^\times$ has only one squareclass (e.g. when $K$ is algebraically closed), we have that 

\begin{cor}
If $K^\times = (K^\times)^2$, then $\Spin(V)$ is a double cover of $SO(V)$.
\end{cor}


Another interesting special case arises when $K=\R$ and $(V,Q)$ is positive definite.  Here all spinor norms are positive, hence they are in the identity squareclass $(\R^\times)^2$, so $\kappa(V) = SO(V)$ and again $\Spin(V)$ is a double cover of $SO(V)$.

%
%
%
%

\section{Spinor Equivalence} \label{Sec:Spinor_equivalence}

In section \ref{Sec:QF_local_global} we have seen that  equivalence of quadratic forms can be viewed as the equivalence of quadratic lattices in a quadratic space $(V, Q)$ by the action of $O(V)$.  There are other more refined notions of equivalence that are useful as well.  For example, equivalence of quadratic lattices under $SO(V)$ is called {\bf proper equivalence}, and plays an essential role in the theory of binary quadratic forms.  In this section we are interested in defining a notion of equivalence called ``spinor equivalence'' that comes from the action of the spin group $\Spin(V)$ and plays an important role for understanding indefinite quadratic forms in $n \geq 3$ variables.

We say that two quadratic forms over $\Z$ are {\bf locally spinor equivalent} if for every place $v$ their associated quadratic lattices are in the same $\kappa(V_v)$-orbit, where $\kappa(V_v)$ is the local spinor kernel group at $v$.

In our definition of the genus $\Gen(Q)$ earlier, we saw that it could also be locally realized by the action of a product of local groups $O(V_v)$ giving local isometries, and in section \ref{Sec:Auto_forms_on_Orthog} we will give a precise adelic version of this statement.  One important thing to check is that the local equivalence defining the genus is weaker than the corresponding global equivalence defining classes.  This is obvious for the definition of $\Gen(Q)$, but must be forcibly imposed in the case of the spinor genus (because $O(V) \not\subseteq \prod_v \kappa(V_v)$).

Suppose that $Q$ has a corresponding quadratic lattice $L \subset (V,Q)$.  Then we define the {\bf spinor genus} of $Q$, denoted $\Spn(Q)$, to be the set of all quadratic forms $Q'$ whose corresponding lattice $L' \subset (V,Q)$ is locally spinor eqiuvalent to $L$ after performing a global isometry (in $O(V)$).  The importance of the spinor genus comes from the following beautiful observation of Eichler that the associated spin groups  have a ``strong approximation'' property, which 
essentially says that the $\Q$-rational points of the ``adelic spin group'' are dense in the adelic group $\Spin(V)_{\A}$. While we avoid a more precise statement here, this adelic formulation of algebraic groups will play a central role in Chapter \ref{Chapter:ThetaLifting}.

\begin{thm}[Eichler \cite{Eichler:1952uq}]  Suppose that $Q$ is a non-degenerate indefinite quadratic form over $\Z$ in $n\geq 3$ variables.  Then there is exactly one class of quadratic forms in its spinor genus.
\end{thm}

\begin{proof}
This is proved in \cite[Thrm 7.1, p186]{Cassels:1978aa} and \cite[Theorem 32.15, p192]{Shimura:2010vn}.
\end{proof}

This theorem allows is to understand statements about indefinite forms in $n\geq 3$ variables by performing various local computations.  There is also a (somewhat modified) version of Siegel's theorem \ref{Thm:Siegel_formulas}  due to Schulze-Pillot \cite{Schulze-Pillot:1984mn} that holds for quadratic forms when one averages over  Spinor genus $\Spn(Q)$ instead of a genus $\Gen(Q)$.  From our computations in section \ref{Sec:local_densities_for_4_squares}, we see that this kind of formula gives direct access to the arithmetic of $Q$ when there is only one class in its  spinor genus.

\chapter{The Theta Lifting}\label{Chapter:ThetaLifting}




\section{Classical to Adelic modular forms for $\GL_{2}$}
It is convenient to 
understand 
the  transformation property (\ref{Eq:modular_form_tranformation_law}) of 
modular forms 
by 
viewing 
them as functions on the algebraic group $\GL_{2}$ with certain invariance properties.  
We do this in two steps, first by lifting the function $f(z)$ on $\H$ to a function $\tilde{f}$ on $\GL_{2}(\R)$, and then by further lifting this to a function $\Fcal$ on the adelic group $\GL_{2}(\A)$ whose transformation properties can be seen most simply.  
This adelic perspective will also give us a very flexible language to use to describe the lifting of modular forms via the Weil representation.  This passage from classical to adelic modular forms is described in \cite[\S3]{Gelbart:1975ab}, \cite[\S10]{Shimura:1997aa} and \cite[\S3.1.5]{Hida2000}.

\medskip
Given a modular form $f(z):\H \ra \C$ of integral weight $k\in \Z$, level $N$, Dirichlet character $\chi$ and trivial multiplier system as in  Definition \ref{defn:modular_form} (so $f \in M_k(N,\chi)$),  we can express its defining transformation property (\ref{Eq:modular_form_tranformation_law}) as the {\it invariance property} $(f|_{k, \chi}\gamma)(z) = f(z)$ for all $\gamma \in \Gamma_0(N)$ with respect to the {\bf weight-character slash operator} 
\begin{equation}
\label{eq:weight_character_slash_operator}
(f|_{k, \chi}\, \gamma)(z) := f(\gamma\cdot z)(cz+d)^{-k}\chi(d)^{-1}, 
\qquad\text{where }
\gamma = 
\[
\begin{smallmatrix}
a & b \\ c & d
\end{smallmatrix}
\] \in \Gamma_0(N).
\end{equation}

To create an invariant function on $\GL^{+}_2(\R)$, we first notice that the weight-character slash operator cannot be extended to allow $\gamma \in \GL^{+}_2(\R)$ because the character factor $\chi(d)$ does not make sense in this generality, though the weight factor $(cz+d)^k$ does make sense.  
%
%
%
%
However if we ignore the character $\chi$ in (\ref{eq:weight_character_slash_operator}) (i.e. take $\chi = 1$ there), then
we do get a well-defined {\bf weight slash operator}
$$
(f|_{k}\, g)(z) := f(g\cdot z)(cz+d)^{-k}, 
\qquad\text{where }
g = 
\[
\begin{smallmatrix}
a & b \\ c & d
\end{smallmatrix}
\] \in \GL^{+}_2(\R).
$$
For these operators the transformation property of $f \in M_k(N,\chi)$ becomes the twisted invariance $(f|_{k}\gamma)(z) = \chi(d) f(z)$ for all $\gamma \in \Gamma_0(N)$.  We can also use this to make a twisted invariant function $\tilde{f}$ on $\GL^{+}_2(\R)$ by noticing that the linear fractional transformation action of $\GL^+_{2}(\R)$ on $\H$ is transitive.  By choosing a distinguished point $i\in \H$, we can define $\tilde{f}:\GL^{+}_2(\R) \ra \C$ by 
$$
\tilde{f}(g) := (f|_{k} \,g)(i),
$$
which satisfies $\tilde{f}(\gamma g) = \chi(d) \tilde{f}(g)$ for all $\gamma \in \Gamma_0(N)$.

%

At this point we do not yet have a truly invariant function on $\GL^{+}_2(\R)$ unless the character $\chi$ is trivial.  To incorporate the character $\chi$ into our formalism, we need to find a natural place for this ``mod $N$ character'' to live.  The necessary congruence structure is provided by the groups $\GL_2(\Q_p)$ at the non-archimedean (i.e. $p$-adic) places, and the natural structure combining all of these completions $\GL_2(\Q_v)$ is the ``adelization'' $\GL_2(\A)$ of the algebraic group $\GL_2$ over $\Q$, defined in the next section.  
For our purposes in this chapter we will not be too interested in the role of the character when passing from a classical modular form to an adelic one, however we will be very interested in adelic modular forms in general as the natural ``invariant'' setting for discussing modular forms.

%

%

\section{Adelizations and Adelic modular forms}\label{Section:Adelic_Modular}

In this section we give a general definition for adelic modular forms for a general algebraic group $G$.  
This agrees with the definitions above when $G = \GL_{2}$, and in future sections we will want to consider $G$ to be either the symplectic group $\Sp_{2n}$ or the special orthogonal group $SO(Q)$ of a definite rational quadratic form $Q$.

\medskip

We first define the {\bf adelization of an affine/linear algebraic group} $G$ over the ring of integers $\O$ of a number field $F$ (defined as the zero set of an ideal of relations in a polynomial ring $\O[\x]$) to be the {\bf restricted direct product} $\prod_{v}' G(F_{v})$ of the local algebraic groups $G(F_{v})$ over all places $v$ of $F$, which is the subset of the usual direct product $\prod_{v} G(F_{v})$ satisfying the restriction that $g_{\A} = (g_{v})_{v}$ is subject to the restriction that $g_{v} \in G(\O_{v})$ for all but finitely many $v$. The restricted direct product has several advantages over the usual direct product -- it is small enough to be locally compact (since every element has all but finitely many components in the compact group $G(\O_{v})$), but it is large enough to contain all rational points $G(F)$.

For the convenience of the reader, we will consistently use subscripts (e.g. $\A, v, \a, \f$) to denote the kind of element (resp. adelic, local, archimedean, non-archimedean/finite) that the element $g_{\bullet}\in G(F_{\bullet})$ represents.  We also denote the center of $G$ by $Z$, to which the same conventions apply for $z\in Z(F_{\bullet})$.  Elements without subscripts will represent rational elements (i.e. we take $g\in G(F)$).  In most cases $Z(F) = F^{\times}$ and $Z_{\A} := Z(F_{\A})$ are the ideles of $F$.  It is also common to denote the compact groups $G(\O_{v})$ as $K_{v}$, with $K_{v}$ denoting a fixed choice of the maximal compact subgroup in $G(F_{v})$ when $v$ is archimedean.


\medskip

In the case where $G = \SL_2$ and $F=\Q$, we can use this notion of an adelic group to further lift a classical modular form $f(z)$ to an ``invariant'' function on $\SL_2(\A)$.  To do this we 
%
extend the original weight-character slash operator to an operator on $\SL_2(\A)$ by writing the Dirichlet character $\chi(d)$ as a product of prime-power characters $\chi_p: \Z_p/p^{\nu_p}\Z_p  \cong \Z/p^{\nu_p}\Z \ra \C$ where $\nu_p := \ord_p(N)$.  These $\chi_p$ can be thought of as characters on the $p$-adic congruence subgroups 
$$
K_p(N) := \{
g = 
\[
\begin{smallmatrix}
a & b \\ c & d
\end{smallmatrix}
\] 
\in \SL_2(\Z_p) \mid c \in p^{\nu_p}\Z_p 
\},
$$
and so $\chi$ can be thought of as a character on the compact product group $K_\f(N) := \prod_p K_p(N)$ by the formula
$$
\chi:(x_p)_{\{p\in\f\}} \longmapsto \prod_p \chi_p(x_p)
$$
in which all but finitely many factors $\chi_p(x_p)=1$.
%
%
%
With this reformulation of the Dirichlet character $\chi$, we define the {\bf adelic slash operator} by the formula
$$
(f|_{k, \chi, \A}\, g_\A)(z) := f(g_\infty\cdot z)(c_\infty z+d_\infty)^{-k} \prod_p\chi_p(d_p)^{-1}, 
\qquad\text{where }
g_\A = 
\[
\begin{smallmatrix}
a & b \\ c & d
\end{smallmatrix}
\] \in \SL_2(\A),
$$
and notice that for all $\gamma \in \Gamma_0(N)$ (considered as an elements of $\SL_2(\A)$ by the canonical diagonal embedding $\gamma\mapsto (\gamma, \gamma, \dots)$) we have the invariance property that 
$$(f|_{k, \chi, \A}\, \gamma)(z) = (f|_{k, \chi}\, \gamma)(z) = f(z).$$

We can now lift the classical modular form $f(z) \in M_k(N, \chi)$ to an adelic function  $\Fcal:\SL_2(\A)\ra \C$ by defining its dependence on $g_\A$ through its action on the distinguished point $i\in \H$ as 
\begin{equation} \label{Eq:F_adelic_and_classical}
\Fcal(g_\A) := (f|_{k, \chi, \A}\, g_\A)(i).
\end{equation}
One can easily verify that this adelic lift $\Fcal$ 
satisfies the following important invariance properties: 
\begin{itemize}
\item $\Fcal(g' g_\A) = \Fcal(g_\A)$ for all $g' \in \SL_2(\Q)$,
\item $\Fcal(g_\A k_\f) = \Fcal(g_\A) \cdot \chi(k_\f)$ for all $k_\f \in K_\f(N)$.
\end{itemize}
One could work a little harder to show that $\Fcal(g_\A)$ is an ``adelic automorphic form'' for the group $G = \SL_2$ in the sense defined below, but for our purposes in these notes the most important properties are the ``rational left-invariance'' and ``right $K$-finiteness'' properties just mentioned.  These are the features of adelic automorphic forms that will be most prominent as we perform our explicit theta-lift.


%





%


\medskip


Since our main goal is to move automorphic forms form one group to another, it will be important to have a definition of automorphic forms that is general enough to cover all cases of interest.  In general, one defines  an {\bf adelic automorphic form} on a linear algebraic group $G$ to be a function $\Fcal:G_{\A} \ra \C$ satisfying:
%
%
%
\begin{enumerate}

\item $\Fcal$ is left-invariant for the rational group: $\Fcal(g \cdot g_{\A}) = \Fcal(g_{\A})$ for all $g \in G(F)$ and for all $g_{\A}\in G_{\A}$.

\item $\Fcal$ has a central adelic (Hecke) character $\psi:Z_{\A}:  Z(\A)\ra \C^\times$ so that  $\Fcal(z_{\A} \cdot g_{\A}) = \psi(z_{\A}) \cdot \Fcal(g_{\A})$ for all $z_{\A}\in Z_{\A}$ and for all $g_{\A}\in G_{\A}$.

\item $\Fcal$ is right-``$K_{\A}$-finite'', meaning that the span of $\Fcal$ as a function under the action of $K_{\A}$ by the right regular representation
$k_{\A}: \Fcal \mapsto \widetilde{\Fcal}(g_{\A}) := \Fcal(g_{\A} \cdot k_{\f})$
%
%
is a finite-dimensional vector space over $F$.
%

\item $\Fcal_{\a}$ is smooth and ``$\mathfrak z_{\a}$-finite'', where $\mathfrak z_{\a}$ is the center of the universal enveloping algebra for $G_{\a}$: meaning that the image of $\Fcal$ under $\mathfrak z$ spans a finite-dimensional vector space over $F_\infty$.  We note that $\mathfrak z_{\a}$ can also be interpreted as the ring of bi-invariant differential operators on $G_{\a}$, and in the case of $\GL_{2}(\R)$ that $\mathfrak z_{\a}$ is the ring $\C[\Delta]$ where $\Delta$ is the hyperbolic Laplacian operator $-y^2\(\frac{\partial^{2}}{\partial x^{2}} + \frac{\partial^{2}}{\partial y^{2}}\)$.

\item $\Fcal$ has moderate growth: meaning that there are constants $C$ and $M \in \R>0$ so that $|\Fcal(\[\begin{smallmatrix} a & 0 \\ 0 & 1\end{smallmatrix}\] g_{\A})| \leq C |a|_\A^{M}$ for all $a \in \A^\times_F$ with $|a|_\A > c$ for some $c$, and all $g_{\A}$ in any fixed compact subset of $G_{\A}$.

\end{enumerate}


The most important conditions for us will be conditions (1)-(3).  Condition (4) is a generalization of the usual holomorphy condition for classical modular forms (since any real-analytic function is an eigenfunction of the Laplacian operator with eigenvalue zero), and condition (5) is a technical growth condition used to exclude poorly behaved functions.
In the case where $G$ is an orthogonal group of a definite quadratic form 
conditions (4) and (5) can be safely omitted because the archimedean component is already compact. 
%
%

\bigskip

{\bf References for this section:} For adelizations of algebraic groups and modular forms for an adelic group see \cite[\S8, 10, 11]{Shimura:1997aa},  for definitions of adelic modular forms on $\GL_{2}$ \cite[\S1.3, pp40-53]{Gelbart:1975ab} and \cite[\S3.1]{Hida2000}.
%
Some motivation for this reformulation of the classical language can be found in the brief Corvallis article of Piatetski-Shapiro \cite{PS1979}. See also Bump's book \cite[\S3.1--2]{Bump1997} for a discussion of the adelic approach to automorphic forms for the important groups $\GL_1$ and $\GL_2$.













\section{The Weil representation}

To see how theta functions arise in terms of representation theory, we now define the {\bf Weil representation} whose symmetries will be closely related to the Fourier transform.  We will not go through the explicit construction of the Weil representation, but instead content ourselves to list its defining properties below and go on to use the Weil representation to produce classical theta functions.  Some references that explicitly construct the Weil representation are \cite{Lion:1980aa}, or Gelbart's book \cite{Gelbart:1976aa}.  Other places where the Weil representation is used in a similar way are \cite{Prasad1993, Prasad1998, MR546603, MR546603}, Kudla's lecture notes \cite{Kudla:2008ls}, and Gelbart's book, \cite[\S7A, pp134--150]{Gelbart:1975aa}.


These considerations give rise to the adelic Weil Representation $\W: \Sp_{2}(F) \bs \Sp_{2}(\A) \ra \GL(S(V_{\A}))$ on the space of Schwartz functions on $V_{\A}$, 
 defined by the following transformation formulas:
%
%
\begin{enumerate}  

\item $\(\W\(\levi{a}\)\Phi\)(\v) = \chi_{V}(a) \cdot |a|_{\A}^{\frac{n}{2}} \cdot \Phi(a \v)$    
\\  

\item $\(\W\(\uni{x}\)\Phi\)(\v) = e_{\A}(x Q(\v)) \cdot \Phi(\v)$

\item $\(\W\(\weyl\)\Phi\)(\v) = \hat{\Phi}(-\v)$   

\end{enumerate}
Here the character $\chi_{V}(\cdot) := (\cdot, (-1)^{n/2} \det(Q))_{F}$ and $e_{\A}$ denotes the adelic exponential, defined by
$$
e_{\A}((x_v)_v) := e^{2\pi i \, x_\infty} \cdot \prod_p e^{-2\pi i \, \mathrm{Frac}_p(x_p)}
$$ 
where $\mathrm{Frac}_p(x_p)  \in \Q/\Z$ is defined as any rational number with $p$-power denominator for which $x_p  \in \mathrm{Frac}_p(x_p) + \Z_p$.  (Notice that the {\bf adelic exponential} is always given by a finite product since any adele $x_\A$ will have all but finitely many $x_p \in \Z_p$.)  The adelic exponential $e_{\A}$ is also used to define the {\bf adelic Fourier transform}
$$
\hat{\Phi}(\w) := \int_{\v\in V_{\A}} e_{\A}(H(\v, \w)) \Phi(\v) \, d_{\A}\v
$$
of any Schwartz function $\Phi \in S(V_{\A})$, where the additive Haar measure $d_{\A}\v$ is normalized so that $\vol(V_\A) = 1$.  
An explicit reference for these formulas are \cite[Thrm 2.22, p37]{Gelbart:1976aa}; also \cite[I.1.6, p3]{Kudla:2008ls} (though there is a minor typo writing $x$ for $a$ in the second formula there).

%
These formulas uniquely define the Weil representation, since any element of  $\Sp_{2}(F_{\A})$ can be expressed as a product of these elements (using the Bruhat decomposition for $\Sp_2=\SL_2$).  They are also visibly trivial on elements of $\Sp_{2}(F)$, because the adelic absolute value $|\cdot|_{\A}$, the rational Hilbert symbol $(\cdot, \cdot)_F := \prod_v (\cdot, \cdot)_{F_v}$, and the adelic exponential $e_{\A}(\cdot)$ are all trivial on rational elements.

\begin{rem}
We have not explicitly defined {\bf (adelic) Schwartz functions on $V_{\A}$}, which are just finite linear combinations of an infinite product $\prod_v \Phi_v(x_v)$ of Schwartz functions $\Phi_v$ on $F_v$ where $\Phi_v$ is the characteristic function of $\Z_p$ at all but finitely many places.  For more details, see \cite[\S3.1, pp256--7]{Bump1997}.
\end{rem}

\section{Theta kernels and theta liftings}
For convenience, we now let $W$ denote the non-degenerate $2$-dimensional symplectic vector space over $F$, and identify $\SL_2(F)$ with $\Sp(W)$.
The Weil representation restricted to our pair $G \times H := \Sp(W) \times O(V)$ incorporates both an invariance under the orthogonal group, and a Fourier transform from the Weyl element.  

To produce theta functions from this, we will need to introduce the familiar classical features of a self-dual function on a lattice.  In the adelic context, our lattice is provided by the rational points $V(F)$ and the adelic self-dual function is taken to be the local product $\phi_{\A}(\v_{\A}) := \prod_{v} \phi_{v}(\v_{v})$  of the familiar Gaussian exponential $\phi_{\infty}(\v) := e^{-\pi Q(\v)}$ at $\infty$ and the characteristic function of the completion of some fixed lattice $L$ on $V$ at all non-archimedean places $p$.  (To fix ideas we can take $F=\Q$ and take the standard lattice $L = \Z^n$, giving 
$\phi_{p}(\v) := \text{characteristic function of $\Z_{p}^{n}$}$.)


Finally we must sum the values of our function $\phi_{\A}$ over the rational lattice $V(F)$, which gives rise to a ``theta distribution''
$$\theta: \phi \mapsto \sum_{\v \in V(F)} \phi(\v)$$
on functions $\phi \in S(V)$.  We will be interested in the behavior of this distribution under the action of the Weil representation, so we define the {\bf theta kernel}
$$
\theta_{\phi}(g,h) := \sum_{\v \in V(F)} (\W(g, h) \phi)(\v) = \sum_{\v \in V(F)} (\W(g) \phi)(h^{-1}\v)
$$  
as the value of the theta distribution under this action.
This theta kernel already feels very similar to a theta series (since we are summing a quadratic Gaussian over a rational lattice, and our choice $\phi_\A$ is supported only on the integral lattice), though it depends on two variables $g \in \Sp_2(\A)$ and $h \in O_Q(\A)$.  However it is more appropriate to think of this as a part of an adelic ``theta machine'' that will allow us to produce many theta series on $\Sp_2 = \SL_2$ after eliminating the orthogonal variable $h$ in some way.


\medskip

We now study the rational invariance properties of the theta kernel, which allow us to think of our a priori ``adelic'' construction as something ``automorphic''.  The main observation is 
%
%
%
\begin{lem}
The theta kernel $\theta_{\phi}(g,h)$ is a function on $\Sp(W) \bs \Sp(W_{\A}) \times O(V) \bs O(V_{\A})$.
\end{lem}
\begin{proof}
At the end of the previous section we noted that the Weil representation transformation 2 is invariant when $x \in F$.  Transformation 1 with $a \in F^\times$ performs a rational scaling of the values, which leaves the rational lattice $V(F)$ invariant, and transformation 3 preserves the theta kernel because the (adelic) Poisson summation formula tells us that the sum of a function on the standard lattice is the same as the sum using its Fourier transform.  Therefore the theta kernel has the rational lattice symmetries of being left-invariant under $\Sp(W)$.  It is also left-invariant under $O(V(F))$ because that action just permutes $V(F)$.  Therefore we have shown that the theta kernel is rationally left-invariant, and so it descends to a function on the rational left cosets as desired.
\end{proof}
%
%
This rational bi-invariance is exactly what allows us to use the theta kernel to move automorphic forms between the orthogonal and symplectic groups.   Given an automorphic form $F(h_{\A})$ on the orthogonal group $O(V_{\A})$, we can define its {\bf theta lift} by the integral 
\begin{equation} \label{Eq:theta_lift_defn}
(\Theta(F))(g_{\A}) := \int_{h_{\A} \in O(V) \bs O(V_{\A})} F(g_{\A}) \, \theta_{\phi}(g_{\A},h_{\A}) \, dh_{\A}
\end{equation}
of $F$ against the theta kernel with respect to the choice of adelic Haar measure $dh_{\A}$ on $O(V_{\A})$ giving the adelic stabilizer $\Stab_{\A}(L) \subset O(V_{\A})$ volume 1.  The theta lift $\Theta(F)(g_{\A})$ formally inherits the symplectic invariance of the theta kernel, and gives an automorphic form on $\Sp(W_{\A})$ when the integral converges.  In the next few sections we will give an explicit example of how this process can be used to produce the classical theta series 
of a positive definite integer-valued quadratic form. 



\section{Some simple automorphic forms on the orthogonal group} \label{Sec:Auto_forms_on_Orthog}

To actually use the theta lift, we must have at our disposal a supply of automorphic forms on the orthogonal quotient $O(V) \bs O(V_{\A})$.  In this section we describe the simplest of these, characteristic functions of a point, which are surprisingly useful for our purposes.  

We begin by giving a classical interpretation of the orthogonal quotient $O(V) \bs O(V_{\A})$.  
%
Given the rational quadratic quadratic space $(V, Q)$, we 
define an action of the adelic orthogonal group $O(V_{\A})$ on the set of all (quadratic) $\O_F$-lattices in $(V,Q)$ by using the following local-global statement for lattices in a rational vector space.

\begin{lem}
There is a natural bijection between lattices $L$ in an $n$-dimensional $F$-rational vector space $V$ and the tuples $(L_p)_p$ of local lattices $L_p \subset V_p := V \otimes_F F_p$ satisfying the property that all but finitely many $L_p$ are equal to $\O_p^n$.  (Here $p$ runs over the set of (non-zero) primes of $F$.)
\end{lem}
\begin{proof}
This is proved in \cite[Lem 21.6, pp102--3]{Shimura:2010vn} and \cite[Thrm 2, p84]{Weil:1967vn}.  The relevant maps in each direction are 
$$
L \mapsto (L_p := L\otimes_{\O_F} \O_p)_p
\qquad
\text{and}
\qquad
(L_p)_p \mapsto L := \bigcap_p \, (V \cap L_p).
$$
\end{proof}

With this lemma, we define an action of $h_\A \in O(V_\A)$ on the lattices in $(V,Q)$ by acting locally on the associated tuple of local lattices:
$$
h_\A : L \longmapsto  (h_p L_p)_p \stackrel{Lemma}{\longleftrightarrow} h_\A L.
$$
This action produces a new tuple of local lattices, which differs from the first tuple at only finitely many places (by the restricted direct product condition on $O(V_\A)$), and so corresponds to a unique lattice in $(V,Q)$.  Notice that this action makes no use of the non-archimedian component $h_\infty$ of $h_\A$.


We now fix a lattice $L$ in $(V,Q)$, and interpret the action of $O(V_\A)$ on $L$ classically.


\begin{lem}
The orbit of $L$ under $O(V_\A)$ is the genus of $L$.
\end{lem}

\begin{proof}
From the definition of the action, we see that the new lattice $L' := h_A L$ is locally isometric to the lattice $L$ at all primes $p$, so it is in the genus of $L$.  Since $L' \subset (V,Q)$ we see they are also isometric at the archimedian place $\infty$, hence $L' \in \Gen(L)$.

Similarly, any lattice in $\Gen(L)$ can be realized as $h_\A L$ by taking $h_p$ to be the element of the orthogonal group carrying $L_p$ to $L'_p$ at the finitely many primes where $L_p \neq L'_p$, and taking all other components $h_v$ as the identity.
\end{proof}

This interpretation can be extended a little further, by trying to describe the classes in the genus $\Gen(L)$ adelically.  If we define the adelic stabilizer 
$$
K_\A := \Stab_\A(L) := \{ h_\A \in O(V_{\A}) \mid h_A L = L\}
$$ 
then we have a bijection
\begin{align*}
O(V_\A) / \Stab_\A(L) \quad &\stackrel{1-1}{\longleftrightarrow} \quad \Gen(L) \\
h_A \qquad\qquad&\longmapsto\qquad h_A L
\end{align*}
Taking this one step further, we have the important bijection 
$$
O(V) \bs O(V_\A) / \Stab_\A(L) \quad \stackrel{1-1}{\longleftrightarrow} \quad \text{classes in $\Gen(L)$}
$$
because two lattices in $(V,Q)$ are in the same class if they are isometric, hence they differ by the action of an element of $O(V)$.  

This finite quotient $O(V) \bs O(V_\A) / \Stab_\A(L)$ corresponding to the classes in the genus can also be thought of as the analogue of the usual upper half-plane $\H$ (for $\SL_2$) for the orthogonal group $O(V)$.  Under this analogy, we see that the analogue of modular functions for $O(V)$ are just functions on this finite set of points (labelled by the classes $L_i$ in $\Gen(L)$).  The simplest of these are the characteristic functions $\Phi_{L_i}$ of each point, and the constant function $1$, both of which we will see play a special role the theory of theta series.  Since the non-archimedian part of $\Stab_\A(L)$ is a compact group, we can easily verify that both of these functions are adelic automorphic forms in the sense of \S\ref{Section:Adelic_Modular}.


%


\section{Realizing classical  theta functions as theta lifts}

We now compute the theta lifting of the characteristic function $\Phi_{L_j}$ of the double coset of $O(V_{\A})$ corresponding to a chosen quadratic lattice $L_j \in \Gen(L)$, with respect to the fixed choice of function $\phi_{\A}(\v_{\A})$ described above.  We will see that this an adelic automorphic form that classically corresponds to a certain multiple of the familiar theta series 
$$
\Theta_{L_j}(z)  
:= \sum_{\v \in L_j} e^{2\pi i Q(\v) z}
= \sum_{m \in \Z \geq 0} r_{L_j}(m) \,e^{2\pi i m z}.
$$

By our definition of the theta lift in (\ref{Eq:theta_lift_defn}), we are interested in computing 
\begin{align}
\Theta(\Phi_{L_j})(g_{\A}) 
& = \int_{O(V) \bs O(V_{\A})} \Phi_{L_j}(h_{\A}) \, \theta_{\phi}(g_{\A},h_{\A}) \, dh_{\A} 
\end{align}
as an automorphic form on $\SL_2 = \Sp_2$.  To do this we first decompose $O(V_{\A})$ as a union of double cosets corresponding to the classes in the genus of $L$ (i.e. with respect to the adelic stabilizer $K_{\A}$), giving
$$
O(V_{\A}) = \bigsqcup_{i \in I} O(V_{F}) \, \al_{i, \A} \,  K_{\A} 
$$
for some fixed choice of representatives $\al_i := \al_{i,\A} \in O(V_\A)$ 
where $\Gen(L) = \bigsqcup_{i\in I} \Cls(L_{i})$ and $L_{i} = \al_{i} L$.  Since the action of $O(V_\A)$ on lattices only depends on the non-archimedean components of $\al_{i, \A}$, to simplify our lives we choose the $\al_{i, \A}$ to have trivial archimedean components $\al_{i, \a} = 1 \in O(V_\a)$.  It will also be convenient to 
define the adelic stabilizers $K_{i,\A} := \Stab_\A(L_i)$ of the other lattices $L_i \in \Gen(L)$.  With this we compute
\begin{align}
\Theta(\Phi_{L_j})(g_{\A}) 
& = \int_{O(V_{F}) \bs O(V_{\A})} \Phi_{L_j}(h_{\A}) \, \theta_{\phi}(g_{\A},h_{\A}) \, dh_{\A}  \\
& = \int_{O(V_{F}) \bs   \bigsqcup_{i \in I} O(V_{F}) \al_i K_{\A}} \Phi_{L_j}(h_{\A}) \, \theta_{\phi}(g_{\A},h_{\A}) \, dh_{\A}  \\
& = \sum_{i\in I} \int_{O(V_{F}) \bs  O(V_{F}) \al_i K_{\A}} \Phi_{L_j}(h_{\A}) \, \theta_{\phi}(g_{\A},h_{\A}) \, dh_{\A}  \\
& = \sum_{i\in I} \int_{O(V_{F}) \bs  O(V_{F}) \al_i K_{\A}\al_{i}^{-1}} \Phi_{L_j}(h_{\A}\al_{i}) \, \theta_{\phi}(g_{\A},h_{\A}\al_{i}) \, dh_{\A}  \\
& = \sum_{i\in I} \int_{O(V_{F}) \bs  O(V_{F})K_{i,\A}} \Phi_{L_j}(h_{\A}\al_{i}) \, \theta_{\phi}(g_{\A},h_{\A}\al_{i}) \, dh_{\A}  \\
& = \sum_{i\in I} \int_{(O(V_{F}) \cap K_{i,\A}) \bs  K_{i,\A}} \Phi_{L_j}(h_{\A}\al_{i}) \, \theta_{\phi}(g_{\A},h_{\A}\al_{i}) \, dh_{\A}  \\
& = \sum_{i\in I} \frac{1}{|\Aut(L_{i})|} \int_{K_{i,\A}} \Phi_{L_j}(h_{\A}\al_{i}) \, \theta_{\phi}(g_{\A},h_{\A}\al_{i}) \, dh_{\A},
\end{align}
where the last step follows because $O(V_{F}) \cap K_{i,\A}$ is the finite group of rational automorphisms $\Aut(L_{i})$, since $K_{i,\A}$ is the adelic stabilizer of $L_{i}$.  (Note: We are also implicitly using the invariance of the left Haar measure $dh_{\A}$ under right multiplication, because the orthogonal group is unimodular.  See \cite[p23]{Weil:1982aa}, \cite[\S14.4, p137]{Voskresenskiui:1998zr} and \cite[\S2, p123]{Ono:1966ys} for a justification of this bi-invariance.)
%

At this point we have ``unfolded'' our integral to the point where it factors as a product of local integrals (since $K_{i,\A}$ is the product of the local stabilizers $K_{i,v}$ over all places $v$), each of which we can try to evaluate separately.  We first notice that for each summand we have an integral over $K_{i,\A} = \al_i K_{\A}\al_{i}^{-1}$, giving 
\begin{align}
\Theta(\Phi_{L_j})(g_{\A}) 
&= \sum_{i\in I} \frac{1}{|\Aut(L_{i})|} \int_{K_{i,\A}} \Phi_{L_j}(h_{\A}\al_{i}) \, \theta_{\phi}(g_{\A},h_{\A}\al_{i}) \, dh_{\A} \\
&= \sum_{i\in I} \frac{1}{|\Aut(L_{i})|} \int_{K_{\A}} \Phi_{L_j}(\al_{i}h_{\A}) \, \theta_{\phi}(g_{\A},\al_{i}h_{\A}) \, dh_{\A} 
\end{align}
whose integrals do not depend on $i\in I$.  To analyze the internal integral, notice that $h_{\A} \in K_{\A}$, giving that 
\begin{align}
\Phi_{L_j}(\al_{i}h_{\A}) \neq 0 
&\iff \al_{i}h_{\A} \in O(V_{F}) \cdot \al_{j} \cdot K_{\A} \\
&\iff \al_{i} \in O(V_{F}) \cdot \al_{j} \cdot K_{\A} \\
&\iff \al_{i} = \al_{j} 
\end{align}
and so all terms with $\al_{i} \neq \al_{j}$ vanish.  Thus
\begin{align}
\Theta(\Phi_{L_j})(g_{\A}) 
&= \frac{1}{|\Aut(L_{j})|} \int_{h_{\A} \in K_{\A}} \theta_{\phi}(g_{\A}, \al_{j} h_{\A}) \, dh_{\A} 
\end{align}

At this point we have to unwind the theta kernel to evaluate the integral, now heavily using the fact that this is a product of local integrals.  We can simplify the non-archimedean orthogonal action with 
the observation that 
\begin{align}
h_p^{-1} \al_{j}^{-1} \v \in L_p \text{ for all primes $p$} 
&\iff \v \in \al_{j} h_p L_p = L_{j,p} \text{ for all $p$} \\
&\iff \v \in L_j.
\end{align}
This together with our choice that $\al_{j,\infty} = 1 \in O(V_\infty)$ gives  
\begin{align}
\Theta(\Phi_{L_j})(g_{\A}) 
&= \frac{1}{|\Aut(L_{j})|} \int_{h_{\A} \in K_{\A}} \theta_{\phi}(g_{\A}, \al_{j} h_{\A}) \, dh_{\A} \\
&= \frac{1}{|\Aut(L_{j})|} \int_{h_{\A} \in K_{\A}} \sum_{\v \in V_{F}} (\W(g_{\A}) \phi_{\A})(h_{\A}^{-1} \al_{j}^{-1} \v) \, dh_{\A} \\
&= \frac{1}{|\Aut(L_{j})|} \int_{h_{\A} \in K_{\A}} \sum_{\v \in L_j} (\W(g_{\A}) \phi_{\A})(h_{\a}^{-1} \v) \, dh_{\A}.
\end{align}

To understand the non-archimedean symplectic action we take advantage of the invariance of the theta lift under $\Sp(W_{F})$ by invoking the ``strong approximation'' property of symplectic groups, (a special case of) which states that
$$
\Sp(W_\A) = \Sp(W) \cdot \Sp(F_\infty)\prod_p \Sp(\O_p).
$$
This means that we can adjust the element $g_{\A}$ (by left-multiplying with some element $g_{F} \in \Sp(W)$) so that its new local components $g_p$ live in $\Sp(\O_p)$ for all primes $p$.  By using the transformation formulas of the Weil representation we see that each component $g_p$ acts trivially on the  
%
%
characteristic function $\phi_{p}(\v)$ of $L_{p}$.  Thus we can express our theta lift as depending only on the archimedean component $g_\infty$ of $g_\A$, giving
%
%
\begin{align}
\Theta(\Phi_{L_j})(g_{\A}) 
&= \frac{1}{|\Aut(L_{j})|} \int_{h_{\A} \in K_{\A}} \sum_{\v \in L_j} (\W(g_{\infty}) \phi_{\infty})(h_{\a}^{-1} \v) \, dh_{\A}. \label{Eq:theta_lift_almost_done}
\end{align}

Since we are interested in the classical modular form $f(z)$ on $\H$ associated the adelic modular form $\Theta(\Phi_L)$, we need only evaluate this on elements $g_\infty \in \SL_2(\R)$ for which $g_\infty \cdot i = z \in \H$.  We notice that when $x,y \in \R$ with $y > 0$, the elements
%
$$
g_{\a, z} := \uni{x} \levi{\sqrt{y}} 
$$
satisfy $g_{\a, z} \cdot i = x+iy \in \H$.  
For these elements $g_{\a, z}$, the action of the Weil representation in (\ref{Eq:theta_lift_almost_done}) 
can be written more explicitly as
\begin{align}
(\W(g_{\a, z}) \phi_{\a}) (h_{\a}^{-1} \v)
&=  \(\W\(\uni{x}\) \W\(\levi{\sqrt{y}}\) \phi_{\a}\) (h_{\a}^{-1} \v)\\
&=  y^\frac{n}{4} \(\W\(\uni{x}\) \phi_{\a}\) (\sqrt{y}\, h_{\a}^{-1} \v)\\
&=  y^\frac{n}{4} e^{2\pi i x Q(\v)} \phi_{\a} (\sqrt{y}\, h_{\a}^{-1} \v)\\
&=  y^\frac{n}{4} e^{2\pi i x Q(\v)} e^{-2\pi Q(\sqrt{y}\, h_{\a}^{-1} \v)}\\
&=  y^\frac{n}{4} e^{2\pi i x Q(\v)} e^{2\pi i \cdot iy Q(h_{\a}^{-1} \v)}\\
&=  y^\frac{n}{4} e^{2\pi i x Q(\v)} e^{2\pi i \cdot iy Q(\v)}\\
&=  y^\frac{n}{4} e^{2\pi i z Q(\v)}.
\end{align}
Substituting this back into (\ref{Eq:theta_lift_almost_done}) gives
\begin{align}
\Theta(\Phi_{L_j})(g_{\a, z}) 
&= \frac{1}{|\Aut(L_{j})|} \int_{h_{\A} \in K_{\A}} \sum_{\v \in L_j} (\W(g_{\a, z}) \phi_{\a})(h_{\a}^{-1}\v) \, dh_{\A} \\
&= \frac{1}{|\Aut(L_{j})|} \int_{h_{\A} \in K_{\A}} \sum_{\v \in L_j} y^\frac{n}{4} e^{2\pi i Q(\v) z} \, dh_{\A} \\
&= \frac{\vol_{\A}(K_\A)}{|\Aut(L_{j})|} \sum_{\v \in L_j} y^\frac{n}{4} e^{2\pi i Q(\v) z} \\
&= \frac{1}{|\Aut(L_{j})|} \sum_{\v \in L_j} y^\frac{n}{4} e^{2\pi i Q(\v) z}. 
\end{align}
Now using the relation (\ref{Eq:F_adelic_and_classical}) with $k = n/2$ and trivial Dirichlet character $\chi$ we have $g_{\a, z}$ has $(cz + d)^{k} = y^{-k/2}$ and can see that $\Theta(\Phi_{L_j})$ corresponds to the classical weight $k$ modular form 
\begin{align}
f(z) :&= \chi(d) (cz+d)^k \cdot \Theta(\Phi_{L_j})(g_{\a, z}) \\
&= y^{-n/4} \cdot \Theta(\Phi_{L_j})(g_{\a, z}) \\
&= \frac{1}{|\Aut(L_{j})|} \sum_{\v \in L_j} e^{2\pi i Q(\v) z}.
\end{align}
But this is just the usual theta series $\Theta_{L_j}(z)$ weighed by the rational factor $\frac{1}{|\Aut(L_{j})|}$, so we have indeed recovered the classical theta function as the theta lift of the characteristic function of the double coset $O(V) \,\al_{j, \A} \, K_\A$ of the adelic orthogonal group corresponding to the lattice $L_j \in \Gen(L)$.

\bibliographystyle{alpha}	
\bibliography{AllReferences__as_of_2009-09-20_copy,more_refs}		

\begin{thebibliography}{EKM08}

\bibitem[:1959]{:1959fk}
Correspondence [signed ``{R}. {L}ipschitz''].
\newblock {\em Ann. of Math. (2)}, 69:247--251, 1959.
\newblock Attributed to A. Weil.

\bibitem[AZ95]{Andrianov:1995kc}
A.~N. Andrianov and V.~G. Zhuravl{\"e}v.
\newblock {\em Modular forms and {H}ecke operators}, volume 145 of {\em
  Translations of Mathematical Monographs}.
\newblock American Mathematical Society, Providence, RI, 1995.
\newblock Translated from the 1990 Russian original by Neal Koblitz.

\bibitem[Bak81]{Bak:1981uq}
Anthony Bak.
\newblock {\em {$K$}-theory of forms}, volume~98 of {\em Annals of Mathematics
  Studies}.
\newblock Princeton University Press, Princeton, N.J., 1981.

\bibitem[BH83]{Benham:1983fb}
J.~W. Benham and J.~S. Hsia.
\newblock Spinor equivalence of quadratic forms.
\newblock {\em J. Number Theory}, 17(3):337--342, 1983.

\bibitem[Bha]{BHARGAVA}
Manjul Bhargava.
\newblock 2009 arizona winter school lecture notes on "the parametrization of
  rings of small rank".

\bibitem[Bum97]{Bump1997}
Daniel Bump.
\newblock {\em Automorphic forms and representations}, volume~55 of {\em
  Cambridge Studies in Advanced Mathematics}.
\newblock Cambridge University Press, Cambridge, 1997.

\bibitem[Cas78]{Cassels:1978aa}
J.~W.~S. Cassels.
\newblock {\em Rational quadratic forms}, volume~13 of {\em London Mathematical
  Society Monographs}.
\newblock Academic Press Inc. [Harcourt Brace Jovanovich Publishers], London,
  1978.

\bibitem[DS05]{Diamond:2005mw}
Fred Diamond and Jerry Shurman.
\newblock {\em A first course in modular forms}, volume 228 of {\em Graduate
  Texts in Mathematics}.
\newblock Springer-Verlag, New York, 2005.

\bibitem[DSP90]{Duke:1990ay}
William Duke and Rainer Schulze-Pillot.
\newblock Representation of integers by positive ternary quadratic forms and
  equidistribution of lattice points on ellipsoids.
\newblock {\em Invent. Math.}, 99(1):49--57, 1990.

\bibitem[Duk97]{Duke:1997ko}
William Duke.
\newblock Some old problems and new results about quadratic forms.
\newblock {\em Notices Amer. Math. Soc.}, 44(2):190--196, 1997.

\bibitem[Eic52]{Eichler:1952uq}
Martin Eichler.
\newblock Die \"{A}hnlichkeitsklassen indefiniter {G}itter.
\newblock {\em Math. Z.}, 55:216--252, 1952.

\bibitem[Eic66]{Eichler:1966rp}
Martin Eichler.
\newblock {\em Introduction to the theory of algebraic numbers and functions}.
\newblock Translated from the German by George Striker. Pure and Applied
  Mathematics, Vol. 23. Academic Press, New York, 1966.

\bibitem[EKM08]{Elman:2008fk}
Richard Elman, Nikita Karpenko, and Alexander Merkurjev.
\newblock {\em The algebraic and geometric theory of quadratic forms},
  volume~56 of {\em American Mathematical Society Colloquium Publications}.
\newblock American Mathematical Society, Providence, RI, 2008.

\bibitem[Eve76]{Eves:1976bs}
Howard Eves.
\newblock {\em An introduction to the history of mathematics}.
\newblock Holt, Rinehart and Winston, fourth edition, 1976.

\bibitem[Gel75a]{Gelbart:1975aa}
Stephen Gelbart.
\newblock {\em Automorphic forms and representations of adele groups}.
\newblock Department of Mathematics, University of Chicago, Chicago, Ill.,
  1975.
\newblock Lecture Notes in Representation Theory.

\bibitem[Gel75b]{Gelbart:1975ab}
Stephen~S. Gelbart.
\newblock {\em Automorphic forms on ad\`ele groups}.
\newblock Princeton University Press, Princeton, N.J., 1975.
\newblock Annals of Mathematics Studies, No. 83.

\bibitem[Gel76]{Gelbart:1976aa}
Stephen~S. Gelbart.
\newblock {\em Weil's representation and the spectrum of the metaplectic
  group}.
\newblock Lecture Notes in Mathematics, Vol. 530. Springer-Verlag, Berlin,
  1976.

\bibitem[Gel79]{MR546603}
Stephen Gelbart.
\newblock Examples of dual reductive pairs.
\newblock In {\em Automorphic forms, representations and {$L$}-functions
  ({P}roc. {S}ympos. {P}ure {M}ath., {O}regon {S}tate {U}niv., {C}orvallis,
  {O}re., 1977), {P}art 1}, Proc. Sympos. Pure Math., XXXIII, pages 287--296.
  Amer. Math. Soc., Providence, R.I., 1979.

\bibitem[Ger08]{Gerstein:2008jh}
Larry~J. Gerstein.
\newblock {\em Basic quadratic forms}, volume~90 of {\em Graduate Studies in
  Mathematics}.
\newblock American Mathematical Society, Providence, RI, 2008.

\bibitem[GS06]{Gille:2006tz}
Philippe Gille and Tam{\'a}s Szamuely.
\newblock {\em Central simple algebras and {G}alois cohomology}, volume 101 of
  {\em Cambridge Studies in Advanced Mathematics}.
\newblock Cambridge University Press, Cambridge, 2006.

\bibitem[Han04]{Hanke:2004xa}
Jonathan Hanke.
\newblock Some recent results about (ternary) quadratic forms.
\newblock In {\em Number theory}, volume~36 of {\em CRM Proc. Lecture Notes},
  pages 147--164. Amer. Math. Soc., Providence, RI, 2004.

\bibitem[Hid00]{Hida2000}
Haruzo Hida.
\newblock {\em Modular forms and {G}alois cohomology}, volume~69 of {\em
  Cambridge Studies in Advanced Mathematics}.
\newblock Cambridge University Press, Cambridge, 2000.

\bibitem[IK04]{Iwaniec:2004la}
Henryk Iwaniec and Emmanuel Kowalski.
\newblock {\em Analytic number theory}, volume~53 of {\em American Mathematical
  Society Colloquium Publications}.
\newblock American Mathematical Society, Providence, RI, 2004.

\bibitem[Iwa87]{Iwaniec:1987tt}
Henryk Iwaniec.
\newblock Spectral theory of automorphic functions and recent developments in
  analytic number theory.
\newblock In {\em Proceedings of the {I}nternational {C}ongress of
  {M}athematicians, {V}ol. 1, 2 ({B}erkeley, {C}alif., 1986)}, pages 444--456,
  Providence, RI, 1987. Amer. Math. Soc.

\bibitem[Iwa97]{Iwaniec:1997ph}
Henryk Iwaniec.
\newblock {\em Topics in classical automorphic forms}, volume~17 of {\em
  Graduate Studies in Mathematics}.
\newblock American Mathematical Society, Providence, RI, 1997.

\bibitem[Jac]{Jacobi.:1829hc}
C.~Jacobi.
\newblock {\em Fundamenta nova theoriae functionum ellipticarum.}
\newblock K{\"o}nigsberg, 1829. {I}n {L}atin. {R}eprinted with corrections in:
  {C.} {J}acobi. {G}esammelte {W}erke. 8 volumes. {B}erlin, 1881-1891. 1.
  49-239. reprinted new york (chelsea, 1969) and available from the american
  mathematical society. edition.

\bibitem[Jac89]{Jacobson:1989fu}
Nathan Jacobson.
\newblock {\em Basic algebra. {II}}.
\newblock W. H. Freeman and Company, New York, second edition, 1989.

\bibitem[Kap03]{Kaplansky:2003kx}
Irving Kaplansky.
\newblock {\em Linear algebra and geometry}.
\newblock Dover Publications Inc., Mineola, NY, revised edition, 2003.
\newblock A second course.

\bibitem[Kne66]{Kneser:1966ly}
Martin Kneser.
\newblock Strong approximation.
\newblock In {\em Algebraic {G}roups and {D}iscontinuous {S}ubgroups ({P}roc.
  {S}ympos. {P}ure {M}ath., {B}oulder, {C}olo., 1965)}, pages 187--196. Amer.
  Math. Soc., Providence, R.I., 1966.

\bibitem[Kno70]{Knopp:1970ul}
Marvin~I. Knopp.
\newblock {\em Modular functions in analytic number theory}.
\newblock Markham Publishing Co., Chicago, Ill., 1970.

\bibitem[Knu91]{Knus:1991qa}
Max-Albert Knus.
\newblock {\em Quadratic and {H}ermitian forms over rings}, volume 294 of {\em
  Grundlehren der Mathematischen Wissenschaften [Fundamental Principles of
  Mathematical Sciences]}.
\newblock Springer-Verlag, Berlin, 1991.
\newblock With a foreword by I. Bertuccioni.

\bibitem[Kob93]{Koblitz:1993vc}
Neal Koblitz.
\newblock {\em Introduction to elliptic curves and modular forms}, volume~97 of
  {\em Graduate Texts in Mathematics}.
\newblock Springer-Verlag, New York, second edition, 1993.

\bibitem[KR88a]{Kudla:1988ab}
Stephen~S. Kudla and Stephen Rallis.
\newblock On the {W}eil-{S}iegel formula.
\newblock {\em J. Reine Angew. Math.}, 387:1--68, 1988.

\bibitem[KR88b]{Kudla:1988aa}
Stephen~S. Kudla and Stephen Rallis.
\newblock On the {W}eil-{S}iegel formula. {II}. {T}he isotropic convergent
  case.
\newblock {\em J. Reine Angew. Math.}, 391:65--84, 1988.

\bibitem[Kud08]{Kudla:2008ls}
Stephen~S. Kudla.
\newblock Some extensions of the {S}iegel-{W}eil formula.
\newblock In {\em Eisenstein series and applications}, volume 258 of {\em
  Progr. Math.}, pages 205--237. Birkh\"auser Boston, Boston, MA, 2008.

\bibitem[Lam05]{Lam:2005kl}
T.~Y. Lam.
\newblock {\em Introduction to quadratic forms over fields}, volume~67 of {\em
  Graduate Studies in Mathematics}.
\newblock American Mathematical Society, Providence, RI, 2005.

\bibitem[Lan94]{Lang:1994fu}
Serge Lang.
\newblock {\em Algebraic number theory}, volume 110 of {\em Graduate Texts in
  Mathematics}.
\newblock Springer-Verlag, New York, second edition, 1994.

\bibitem[Lan95]{Lang:1995kx}
Serge Lang.
\newblock {\em Algebra}.
\newblock Addison-Wesley Publishing Company, Inc., Reading, MA, third edition,
  1995.

\bibitem[LV80]{Lion:1980aa}
G{\'e}rard Lion and Mich{\`e}le Vergne.
\newblock {\em The {W}eil representation, {M}aslov index and theta series},
  volume~6 of {\em Progress in Mathematics}.
\newblock Birkh\"auser Boston, Mass., 1980.

\bibitem[Miy06]{Miyake:2006hf}
Toshitsune Miyake.
\newblock {\em Modular forms}.
\newblock Springer Monographs in Mathematics. Springer-Verlag, Berlin, english
  edition, 2006.
\newblock Translated from the 1976 Japanese original by Yoshitaka Maeda.

\bibitem[MW06]{Moreno:2006qf}
Carlos~J. Moreno and Samuel~S. Wagstaff, Jr.
\newblock {\em Sums of squares of integers}.
\newblock Discrete Mathematics and its Applications (Boca Raton). Chapman \&
  Hall/CRC, Boca Raton, FL, 2006.

\bibitem[Neu99]{Neukirch:1999mi}
J{\"u}rgen Neukirch.
\newblock {\em Algebraic number theory}, volume 322 of {\em Grundlehren der
  Mathematischen Wissenschaften [Fundamental Principles of Mathematical
  Sciences]}.
\newblock Springer-Verlag, Berlin, 1999.
\newblock Translated from the 1992 German original and with a note by Norbert
  Schappacher, With a foreword by G. Harder.

\bibitem[NSW08]{Neukirch:2008pi}
J{\"u}rgen Neukirch, Alexander Schmidt, and Kay Wingberg.
\newblock {\em Cohomology of number fields}, volume 323 of {\em Grundlehren der
  Mathematischen Wissenschaften [Fundamental Principles of Mathematical
  Sciences]}.
\newblock Springer-Verlag, Berlin, second edition, 2008.

\bibitem[O'M71]{OMeara:1971zr}
O.~T. O'Meara.
\newblock {\em Introduction to quadratic forms}.
\newblock Springer-Verlag, New York, 1971.
\newblock Second printing, corrected, Die Grundlehren der mathematischen
  Wissenschaften, Band 117.

\bibitem[Ono66]{Ono:1966ys}
Takashi Ono.
\newblock On {T}amagawa numbers.
\newblock In {\em Algebraic {G}roups and {D}iscontinuous {S}ubgroups ({P}roc.
  {S}ympos. {P}ure {M}ath., {B}oulder, {C}olo., 1965)}, pages 122--132. Amer.
  Math. Soc., Providence, R.I., 1966.

\bibitem[Par]{Parimala}
Raman Parimala.
\newblock 2009 arizona winter school lecture notes on "some aspects of the
  algebraic theory of quadratic forms".

\bibitem[Pfe71]{Pfeuffer:1971pd}
Horst Pfeuffer.
\newblock Einklassige {G}eschlechter totalpositiver quadratischer {F}ormen in
  totalreellen algebraischen {Z}ahlk\"orpern.
\newblock {\em J. Number Theory}, 3:371--411, 1971.

\bibitem[Pfe78]{Pfeuffer:1977ve}
Horst Pfeuffer.
\newblock Darstellungsmasse bin\"arer quadratischer {F}ormen \"uber
  totalreellen algebraischen {Z}ahlk\"orpern.
\newblock {\em Acta Arith.}, 34(2):103--111, 1977/78.

\bibitem[Pie82]{Pierce:1982fk}
Richard~S. Pierce.
\newblock {\em Associative algebras}, volume~88 of {\em Graduate Texts in
  Mathematics}.
\newblock Springer-Verlag, New York, 1982.
\newblock Studies in the History of Modern Science, 9.

\bibitem[PR94]{Platonov:1994ve}
Vladimir Platonov and Andrei Rapinchuk.
\newblock {\em Algebraic groups and number theory}, volume 139 of {\em Pure and
  Applied Mathematics}.
\newblock Academic Press Inc., Boston, MA, 1994.
\newblock Translated from the 1991 Russian original by Rachel Rowen.

\bibitem[Pra93]{Prasad1993}
Dipendra Prasad.
\newblock Weil representation, {H}owe duality, and the theta correspondence.
\newblock In {\em Theta functions: from the classical to the modern}, volume~1
  of {\em CRM Proc. Lecture Notes}, pages 105--127. Amer. Math. Soc.,
  Providence, RI, 1993.

\bibitem[Pra98]{Prasad1998}
Dipendra Prasad.
\newblock A brief survey on the theta correspondence.
\newblock In {\em Number theory ({T}iruchirapalli, 1996)}, volume 210 of {\em
  Contemp. Math.}, pages 171--193. Amer. Math. Soc., Providence, RI, 1998.

\bibitem[PS79]{PS1979}
I.~I. Piatetski-Shapiro.
\newblock Classical and adelic automorphic forms. {A}n introduction.
\newblock In {\em Automorphic forms, representations and {$L$}-functions
  ({P}roc. {S}ympos. {P}ure {M}ath., {O}regon {S}tate {U}niv., {C}orvallis,
  {O}re., 1977), {P}art 1}, Proc. Sympos. Pure Math., XXXIII, pages 185--188.
  Amer. Math. Soc., Providence, R.I., 1979.

\bibitem[Sah60]{Sah:1960kx}
Chih-han Sah.
\newblock Quadratic forms over fields of characteristic {$2$}.
\newblock {\em Amer. J. Math.}, 82:812--830, 1960.

\bibitem[Sal99]{Saltman:1999or}
David~J. Saltman.
\newblock {\em Lectures on division algebras}, volume~94 of {\em CBMS Regional
  Conference Series in Mathematics}.
\newblock Published by American Mathematical Society, Providence, RI, 1999.

\bibitem[Ser77]{Serre:1977cr}
J.-P. Serre.
\newblock Modular forms of weight one and {G}alois representations.
\newblock In {\em Algebraic number fields: {$L$}-functions and {G}alois
  properties ({P}roc. {S}ympos., {U}niv. {D}urham, {D}urham, 1975)}, pages
  193--268. Academic Press, London, 1977.

\bibitem[SH98]{Scharlau:1998jo}
Rudolf Scharlau and Boris Hemkemeier.
\newblock Classification of integral lattices with large class number.
\newblock {\em Math. Comp.}, 67(222):737--749, 1998.

\bibitem[Shi73]{Shimura:1973aa}
Goro Shimura.
\newblock On modular forms of half integral weight.
\newblock {\em Ann. of Math. (2)}, 97:440--481, 1973.

\bibitem[Shi94]{Shimura:1994ab}
Goro Shimura.
\newblock {\em Introduction to the arithmetic theory of automorphic functions},
  volume~11 of {\em Publications of the Mathematical Society of Japan}.
\newblock Princeton University Press, Princeton, NJ, 1994.
\newblock Reprint of the 1971 original, Kano Memorial Lectures, 1.

\bibitem[Shi97]{Shimura:1997aa}
Goro Shimura.
\newblock {\em Euler products and {E}isenstein series}, volume~93 of {\em CBMS
  Regional Conference Series in Mathematics}.
\newblock Published for the Conference Board of the Mathematical Sciences,
  Washington, DC, 1997.

\bibitem[Shi04]{Shimura:2004qe}
Goro Shimura.
\newblock {\em Arithmetic and analytic theories of quadratic forms and
  {C}lifford groups}, volume 109 of {\em Mathematical Surveys and Monographs}.
\newblock American Mathematical Society, Providence, RI, 2004.

\bibitem[Shi06a]{Shimura:2006ai}
Goro Shimura.
\newblock Integer-valued quadratic forms and quadratic {D}iophantine equations.
\newblock {\em Doc. Math.}, 11:333--367 (electronic), 2006.

\bibitem[Shi06b]{Shimura:2006tg}
Goro Shimura.
\newblock Quadratic {D}iophantine equations, the class number, and the mass
  formula.
\newblock {\em Bull. Amer. Math. Soc. (N.S.)}, 43(3):285--304 (electronic),
  2006.

\bibitem[Shi10]{Shimura:2010vn}
Goro Shimura.
\newblock {\em Arithmetic of quadratic forms}.
\newblock Springer Monographs in Mathematics. Springer, New York, 2010.

\bibitem[Sie35]{Siegel:1935qf}
Carl~Ludwig Siegel.
\newblock \"{U}ber die analytische {T}heorie der quadratischen {F}ormen.
\newblock {\em Ann. of Math. (2)}, 36(3):527--606, 1935.

\bibitem[Sie36]{Siegel:1936qo}
Carl~Ludwig Siegel.
\newblock \"{U}ber die analytische {T}heorie der quadratischen {F}ormen. {II}.
\newblock {\em Ann. of Math. (2)}, 37(1):230--263, 1936.

\bibitem[Sie37]{Siegel:1937jw}
Carl~Ludwig Siegel.
\newblock \"{U}ber die analytische {T}heorie der quadratischen {F}ormen. {III}.
\newblock {\em Ann. of Math. (2)}, 38(1):212--291, 1937.

\bibitem[Sie63]{Siegel:1963yt}
Carl~Ludwig Siegel.
\newblock {\em Lectures on the analytical theory of quadratic forms}.
\newblock Notes by Morgan Ward. Third revised edition. Buchhandlung Robert
  Peppm\"uller, G\"ottingen, 1963.

\bibitem[SP84]{Schulze-Pillot:1984mn}
Rainer Schulze-Pillot.
\newblock Darstellungsma\ss e von {S}pinorgeschlechtern tern\"arer
  quadratischer {F}ormen.
\newblock {\em J. Reine Angew. Math.}, 352:114--132, 1984.

\bibitem[SP00]{Schulze-Pillot:2000fm}
Rainer Schulze-Pillot.
\newblock Exceptional integers for genera of integral ternary positive definite
  quadratic forms.
\newblock {\em Duke Math. J.}, 102(2):351--357, 2000.

\bibitem[SP04]{Schulze-Pillot:2004ir}
Rainer Schulze-Pillot.
\newblock Representation by integral quadratic forms---a survey.
\newblock In {\em Algebraic and arithmetic theory of quadratic forms}, volume
  344 of {\em Contemp. Math.}, pages 303--321. Amer. Math. Soc., Providence,
  RI, 2004.

\bibitem[Str09]{Strang_linear_algebra_4}
Gilbert Strang.
\newblock {\em Introduction to Linear Algebra}.
\newblock Wellesley Cambridge Press, New York, fourth edition, 2009.

\bibitem[Tar29]{Tartakowsky:1929dq}
W.~Tartakowsky.
\newblock Die gesamtheit der zahlen, die durch eine positive quadratische form
  $f(x_1,x_2,\dots,x_s)$ $(s \geq 4)$ darstellbar sind. i, ii.
\newblock {\em Bull. Ac. Sc. Leningrad}, 2(7):111--122; 165--196, 1929.

\bibitem[Tor05]{Tornaria:2005sy}
Gonzalo Tornaria.
\newblock {\em The Brandt module of ternary quadratic lattices}.
\newblock PhD thesis, University of Texas, Austin, 2005.

\bibitem[Voi]{VOIGHT}
John Voight.
\newblock Computing with quaternion algebras: Identifying the matrix ring.

\bibitem[Vos98]{Voskresenskiui:1998zr}
V.~E. Voskresenski{\u\i}.
\newblock {\em Algebraic groups and their birational invariants}, volume 179 of
  {\em Translations of Mathematical Monographs}.
\newblock American Mathematical Society, Providence, RI, 1998.
\newblock Translated from the Russian manuscript by Boris Kunyavski [Boris
  {\`E}. Kunyavski{\u\i}].

\bibitem[Wat63]{Watson:1963aa}
G.~L. Watson.
\newblock One-class genera of positive quadratic forms.
\newblock {\em J. London Math. Soc.}, 38:387--392, 1963.

\bibitem[Wat84]{Watson:1984aa}
G.~L. Watson.
\newblock One-class genera of positive quadratic forms in seven variables.
\newblock {\em Proc. London Math. Soc. (3)}, 48(1):175--192, 1984.

\bibitem[Wei67]{Weil:1967vn}
Andr{\'e} Weil.
\newblock {\em Basic number theory}.
\newblock Die Grundlehren der mathematischen Wissenschaften, Band 144.
  Springer-Verlag New York, Inc., New York, 1967.

\bibitem[Wei82]{Weil:1982aa}
Andr{\'e} Weil.
\newblock {\em Adeles and algebraic groups}, volume~23 of {\em Progress in
  Mathematics}.
\newblock Birkh\"auser Boston, Mass., 1982.
\newblock With appendices by M. Demazure and Takashi Ono.

\end{thebibliography}
\addcontentsline{toc}{chapter}{Bibliography}

\printindex

\end{document}